%% file: verifiability-d2-vms-rom-revision.tex
\documentclass{amsart}
\usepackage[foot]{amsaddr}
\makeatletter
\renewcommand{\email}[2][]{%
  \ifx\emails\@empty\relax\else{\g@addto@macro\emails{,\space}}\fi%
  \@ifnotempty{#1}{\g@addto@macro\emails{\textrm{(#1)}\space}}%
  \g@addto@macro\emails{#2}%
}
\makeatother

\usepackage{amsmath,amsfonts,amssymb}
\usepackage{mathtools} 
\usepackage{mathrsfs} 
\usepackage[active]{srcltx}  
\usepackage[pdftex,bookmarks=true]{hyperref}  
\usepackage{graphicx}
 \usepackage{enumerate}
\usepackage{macros}
\usepackage{lipsum}
\usepackage{bbm}
\usepackage{lipsum}
\usepackage{amsfonts}
\usepackage{enumitem}
\usepackage{amsmath}
\usepackage{graphicx}
\usepackage{epstopdf}

\usepackage[ruled,vlined]{algorithm2e}
\usepackage{algorithmic}

\usepackage{color}
\usepackage{hyperref}
\usepackage{pgfplots}
\usepackage{verbatim}
\usepackage{graphicx}
\usepackage{placeins}
\usepackage{float}

\ifpdf
  \DeclareGraphicsExtensions{.eps,.pdf,.png,.jpg}
\else
  \DeclareGraphicsExtensions{.eps}
\fi

\usepackage{tikz}
\usepackage{fullpage}
\usetikzlibrary{fit,positioning}
\usepackage{siunitx}
\usepackage{float}

\input{notation}

\newtheorem{remark}{Remark}[section]
\newtheorem{lemma}{Lemma}[section]
\newtheorem{theorem}{Theorem}[section]

\newtheorem{definition}{Definition}[section]
\numberwithin{equation}{section}

\definecolor{greenrb}{rgb}{0.2,0.6,0.2}

\title{Verifiability 
of the Data-Driven Variational Multiscale \\
Reduced Order Model}

\author{Birgul Koc}
\address[BK]{
Departamento EDAN \& IMUS, 
Universidad de Sevilla,
Campus de Reina Mercedes, 41012, Sevilla, Spain
  }
\email[BK]{ birkoc@alum.us.es
}

\author{Changhong Mou}
\address[CM]{
Department of Mathematics, University of Wisconsin, Madison, WI 53706, USA 
  }
\email[CM]{ cmou3@wisc.edu
}

\author{Honghu Liu}
\address[HL]{Department of Mathematics, Virginia Tech, Blacksburg, VA 24061, USA}
\email[HL]{hhliu@vt.edu}

\author{Zhu Wang}
\address[ZW]{Department of Mathematics, University of South Carolina, Columbia, SC 29208, USA}
\email[ZW]{wangzhu@math.sc.edu}

\author{Gianluigi Rozza}
\address[GR]{mathLab, Mathematics Area, SISSA, 
I-34136 Trieste, Italy }
\email[GR]{gianluigi.rozza@sissa.it}

\author{Traian Iliescu}
\address[TI]{Department of Mathematics, Virginia Tech, Blacksburg, VA 24061, USA}
\email[TI]{iliescu@vt.edu}

\begin{document}

\maketitle

\begin{abstract}
In this paper, we focus on the mathematical foundations of reduced order model (ROM) closures.
First, we extend the verifiability concept from large eddy simulation to the ROM setting.
Specifically, we call a ROM closure model verifiable if a small ROM closure model error (i.e., a small difference between the true ROM closure and the modeled ROM closure) implies a small ROM error.
Second, we prove that {the} data-driven ROM closure {studied here} (i.e., the data-driven variational multiscale ROM) is verifiable.
Finally, we investigate the verifiability of the data-driven variational multiscale ROM in the numerical simulation of the one-dimensional Burgers equation and a two-dimensional flow past a circular cylinder at Reynolds numbers $Re=100$ and $Re=1000$.
\end{abstract}

\section{Introduction}

Full order models (FOMs) are computational models obtained with classical numerical methods (e.g., finite element or finite difference methods).
In the numerical simulation of fluid flows, FOMs often yield high-dimensional (e.g., $\cO(10^6)$) systems of equations. 
Thus, the computational cost of using FOMs in important many-query fluid flow applications (e.g., uncertainty quantification, optimal control, and shape optimization) can be prohibitively high.

Reduced order models (ROMs) are computational models that yield systems of equations whose dimensions are dramatically lower than those corresponding to FOMs.
For example, in the numerical simulation of fluid flows that are dominated by recurrent spatial structures (e.g., flow past bluff bodies), the dimensions of the resulting system of equations can be $\cO(10)$ for ROMs and $\cO(10^6)$ for FOMs, while the ROM and FOM accuracy is of the same order.
Thus, ROMs have been used in many-query fluid flow applications to reduce the computational cost of FOMs.
Probably the most popular type of ROM used in these applications is the Galerkin ROM (G-ROM), which is constructed by using the Galerkin method.
The G-ROM is based on a simple yet powerful idea:
Instead of using millions or even billions of general purpose basis functions (as in classical  Galerkin methods, such as the tent functions in the finite element method), G-ROM uses a lower-dimensional data-driven basis.
Specifically, the available numerical or experimental data is used to build a few ROM basis functions that model the spatial structures that dominate the flow dynamics.

The G-ROM has been successful in the efficient numerical simulation of relatively simple laminar flows, e.g., flow past a circular cylinder at low Reynolds numbers.
However, the standard G-ROM generally fails in the numerical simulation of turbulent flows.
The main reason is that, in order to ensure a relatively low computational cost, only a few ROM basis functions are used to build the standard G-ROM.
These few ROM basis functions can represent the simple dynamics of laminar flows, but not the complex dynamics of turbulent flows.
Thus, in the numerical simulation of turbulent flows, the standard G-ROM is equipped with a ROM closure model, i.e., a correction term that models the effect of the discarded ROM basis functions on the ROM dynamics.

Over the last two decades, ROM closure modeling has witnessed a dynamic development.
A survey of current ROM closure modeling strategies is presented in~\cite{ahmed2021closures}.  
Three main types of ROM closure models have been proposed: 
(i) Functional ROM closures are constructed by using physical insight. 
Classical examples of functional ROM closures include eddy viscosity models~\cite{wang2012proper}, in which the main role of the ROM closure model is to dissipate energy.
(ii) Structural ROM closures are a different class of models that are developed by using mathematical arguments.  
Examples of structural ROM closures include the approximate deconvolution ROM~\cite{xie2017approximate}, the Mori-Zwanzig formalism~\cite{chorin2015discrete,lu2020data,parish2017unified}, and the parameterizing manifolds~\cite{chekroun2019variational,CLM21_BE,chekroun2014stochastic}.
(iii) The most active research area in ROM closure modeling is in the development of data-driven ROM closures in which available data is utilized to build the ROM closure model.  
An example of data-driven ROM closure is the data-driven variational multiscale ROM (DD-VMS-ROM) that was proposed in~\cite{mou2021data,xie2018data}.
The DD-VMS-ROM has been investigated numerically in~\cite{koc2019commutation,mohebujjaman2019physically,mou2020data,mou2021data,xie2018data,xie2020closure}.
However, providing mathematical support for the DD-VMS-ROM is an open problem.

In classical CFD, there exists extensive mathematical support for closure modeling.
For example, the monographs~\cite{BIL05,john2004large,rebollo2014mathematical} present the mathematical analysis for many large eddy simulation (LES) models, as well as the numerical analysis of their discretization.
In contrast, despite the recent increased interest in ROM closure modeling~\cite{ahmed2021closures}, the mathematical foundations of ROM closures are relatively scarce. 
Indeed, the ROM closure models are generally assessed heuristically:
The proposed ROM closure model is used in numerical simulations and is shown to improve the numerical accuracy of the standard G-ROM and/or other ROM closure models.
However, fundamental questions in ROM numerical analysis are still wide open for most of these ROM closure models:
Is the proposed ROM closure model stable?
Does the ROM closure model converge?
If so, what does it converge to?

Only the first steps in the numerical  analysis  of ROM closures have been taken.
To our knowledge, the first numerical analysis of a ROM closure model was performed in~\cite{borggaard2011artificial}, where an eddy viscosity ROM closure model (i.e., the Smagorinsky model) was analyzed in a simplified setting. 
Next, the numerical analysis of eddy viscosity variational multiscale ROMs was carried out in~\cite{iliescu2013variational,iliescu2014variational}.  
Finally, the numerical analysis of the Samagorinsky model in a reduced basis method (RBM) setting was performed in~\cite{ballarin2020certified,rebollo2017certified}.
We note that numerical analysis for regularized ROMs, which are related to but different from ROM closures, was performed in~\cite{gunzburger2020leray,xie2018numerical}; {see} also~\cite{azaiez2021cure} for related work.

In this paper, we take a next step in the development of numerical analysis for ROM closures and prove verifiability for a data-driven ROM closure model, i.e., the DD-VMS-ROM proposed in~\cite{mou2021data,xie2018data}. 
Specifically, we show that the ROM closure model in the DD-VMS-ROM is accurate in a precise sense.
More importantly, we prove that the DD-VMS-ROM is verifiable, i.e., we prove that since the DD-VMS-ROM closure model is accurate, the DD-VMS-ROM solution is accurate.
We note that this is not a trivial task:
The Navier-Stokes equations (and their filtered counterparts), which are the mathematical models that {we} use in this paper, are nonlinear and sensitive to perturbations, so adding to them a relatively small term (i.e., the ROM closure term) does not automatically imply that the resulting solution will be close to the original one.
To prove that the DD-VMS-ROM closure model is verifiable, we use the following ingredients:
(i) We use ROM spatial filtering to determine an explicit formula for the exact ROM closure term, which needs to be modeled.
(ii) We use data-driven modeling to construct the DD-VMS-ROM closure model and show that this closure model is accurate, i.e., it is close to the exact ROM closure model.
(iii) We use physical constraints to increase the accuracy of our data-driven ROM closure model.
We note that the verifiability concept was defined in an LES context (see, e.g.,~\cite{kaya2002verifiability} as well as~\cite{BIL05} for a survey). 
However, to our knowledge, this is the first time the verifiability concept is defined and investigated in a ROM context.

The rest of the paper is organized as follows:
In Section~\ref{sec:g-rom}, we outline the construction of the standard G-ROM.
In Sections~\ref{sec:les-rom} and \ref{sec:dd-vms-rom}, we use ROM spatial filtering to build LES-ROMs and utilize data-driven modeling to build the closure model in the DD-VMS-ROM, respectively.
In Section~\ref{sec:dd-vms-rom-verifiability}, we prove the main theoretical result in this paper, i.e., we prove that the DD-VMS-ROM is verifiable.
In Section \ref{sec:num}, we illustrate the theoretical developments.
Specifically, for the Burgers equation and the two-dimensional flow past a circular cylinder, we show the following: 
(i) the ROM closure error (i.e., the difference between the true ROM closure term and the DD-VMS-ROM closure term) is small and it becomes smaller and smaller as we increase the ROM dimension; and 
(ii) as the ROM closure error decreases, so does the ROM error (i.e., the DD-VMS-ROM is verifiable). 
Finally, in Section~\ref{sec:conclusions}, we present the conclusions of our theoretical and numerical investigations and outline several directions for future research.

\section{Galerkin ROM (G-ROM)}
    \label{sec:g-rom}

In this section, we outline the construction of the Galerkin ROM (G-ROM) for the Navier-Stokes equations (NSE):
 \begin{eqnarray}
     && \frac{\partial \boldsymbol u}{\partial t}
     - Re^{-1} \Delta \boldsymbol u
     + \boldsymbol u \cdot \nabla \boldsymbol u
     + \nabla p
     = {\boldsymbol f} ,
     \label{eqn:nse-1} \\
     && \nabla \cdot \boldsymbol u
     = 0 ,
     \label{eqn:nse-2}
 \end{eqnarray}
 where $\boldsymbol u$ is the velocity, $p$ the pressure, and $Re$ the Reynolds number.
The NSE~\eqref{eqn:nse-1}--\eqref{eqn:nse-2} are equipped with an initial condition and, for simplicity, homogeneous Dirichlet boundary conditions.
To build the ROM basis, we assume that we have access to the snapshots $\{\boldsymbol u_h^0, ... , \boldsymbol u_h^M \}$, which are the coefficient vectors of the FEM approximations of the NSE~\eqref{eqn:nse-1}--\eqref{eqn:nse-2} at the time instances $t_0, t_1, \ldots, t_M$, respectively.
The number of snapshots, $M$, is an arbitrary positive integer.
In what follows, we assume that $M$ is fixed.
Next, we use these snapshots and the proper orthogonal decomposition (POD)~\cite{HLB96,volkwein2013proper} to construct an orthonormal ROM basis $\{ \boldsymbol \varphi_1, ... , \boldsymbol \varphi_d\}$, which generates the ROM space $\bX^d$ defined as follows:
 \begin{eqnarray} \label{eqn:d-dim-set}
 \bX^d:= \text{span} \{ \boldsymbol \varphi_1, ... , \boldsymbol \varphi_d\} , 
 \end{eqnarray}
 where $d$ is the number of linearly independent snapshots $\{\boldsymbol u_h^0, ... , \boldsymbol u_h^M \}$.
Thus, $d$ is the maximal dimension of a basis that spans the same space as the space spanned by the given snapshots.
 By using the ROM basis functions in \eqref{eqn:d-dim-set}, we construct $\boldsymbol u_d$, which is the $d$-dimensional ROM approximation of NSE velocity, $\bu$:
 \begin{eqnarray}
 \label{eqn:ud-representation}
   \displaystyle \boldsymbol u_d(\bx,t) =  \sum_{i=1}^{d}  ( \boldsymbol a_d)_i(t) \, \boldsymbol \varphi_i(\bx) .
 \end{eqnarray}
To find the vector of ROM coefficients $\ba_d$ in~\eqref{eqn:ud-representation}, we use the Galerkin projection, i.e., we replace $\bu$ with $\bu_d$ in the  NSE~\eqref{eqn:nse-1}--\eqref{eqn:nse-2}, and then project the resulting equations onto the ROM space, $\bX^d$.
This yields the $d$-dimensional Galerkin ROM (G-ROM):
\begin{eqnarray}  
\begin{aligned}
     ( (\boldsymbol u_d)_t, \boldsymbol v_d ) 
     + Re^{-1} (\nabla \boldsymbol u_d, \nabla \boldsymbol v_d )
      + (\boldsymbol u_d \cdot\nabla \boldsymbol u_d, \boldsymbol v_d ) 
     = (\boldsymbol f, \boldsymbol v_d ), \qquad \forall \, \boldsymbol v_d \in \boldsymbol X^d , 
     \end{aligned}
     \label{eqn:d-dimensional-g-rom}
 \end{eqnarray}
{
where $( \cdot , \cdot )$ denotes the $L^2$ inner product.
}
{
We note that the G-ROM~\eqref{eqn:d-dimensional-g-rom}
does not include a pressure term, since the ROM basis functions are assumed to be discretely divergence-free.
This is the case if, e.g., the snapshots are discretely divergence-free. 
Indeed, when POD is used to construct the ROM basis (as in our numerical investigation), the ROM basis functions are linear combinations of the snapshots.
Since the snapshots are discretely divergence-free, so are the ROM basis functions.
We also note that alternative 
formulations within the RBM framework are used in, e.g.,~\cite{ali2020stabilized,ballarin2015supremizer,hess2020reduced,hesthaven2015certified,martini2018certified,quarteroni2015reduced}.
}

By using the backward Euler time discretization, we get the full discretization of the $d$-dimensional G-ROM~\eqref{eqn:d-dimensional-g-rom} as follows: 
$\forall \, n=1,...,M$
 \begin{eqnarray}  
\begin{aligned}
     \left( \frac{\boldsymbol u_d^{n}-\boldsymbol u_d^{n-1}}{\Delta t}, \boldsymbol v_d \right) 
     + Re^{-1} (\nabla \boldsymbol u_d^n, \nabla \boldsymbol v_d )
    + (\boldsymbol u_d^n\cdot\nabla  \boldsymbol u_d^n, \boldsymbol v_d ) 
     = (\boldsymbol f^n, \boldsymbol v_d ),  \quad  \forall \, \boldsymbol v_d \in \boldsymbol X^d , 
     \end{aligned}
     \label{eqn:d-dimensional-g-rom-full-disc}
 \end{eqnarray}
where the superscript $n$ denotes the approximation at time step $n$.
To obtain the finite-dimensional representation of the $d$-dimensional G-ROM~\eqref{eqn:d-dimensional-g-rom-full-disc}, we choose
$\boldsymbol v_d$ to be $\bphi_1,\ldots,\bphi_d$, which yields the following system of equations: 
    \vspace*{-0.1cm}
    \begin{eqnarray}
	    \frac{\boldsymbol a_d^n - \boldsymbol a_d^{n-1}}{\Delta t }
	    = \boldsymbol b^n + \boldsymbol A \, \boldsymbol a_d^n + (\boldsymbol a_d^n)^{\top} \, \boldsymbol B \, \boldsymbol a_d^n ,
    \label{eqn:g-rom-nse}
	    \\[-0.6cm]
	    \nonumber	   
    \end{eqnarray}
    where 
   $\boldsymbol a_d^n$ is the vector of unknown ROM coefficients,  $\boldsymbol b$ is a $d \times 1$ vector,   
    $\boldsymbol A$ is a $d \times d$ matrix, and $\boldsymbol B$ is a $d \times d \times d$ tensor.
The system of equations in~\eqref{eqn:g-rom-nse} can be written componentwise as follows:
    \begin{eqnarray}
 \frac{(\boldsymbol a_d^n)_i - (\boldsymbol a_d^{n-1})_i}{\Delta t } = {\boldsymbol b^n_i} + \sum_{m=1}^d \boldsymbol A_{im} \boldsymbol a_m^n
     +\sum_{m=1}^d \sum_{k=1}^d \boldsymbol B_{imk} \boldsymbol a_m^n \boldsymbol a_k^n\,,
     \quad 
     \, 1 \leq i \leq d \,,  
    \end{eqnarray}
    where, for $1 \leq i, m, k \leq d $,
    \begin{align} 
        \boldsymbol b_i^n &= ( \boldsymbol f^n, \boldsymbol \varphi_i), \label{eqn:g-rom_operators-b}\\
	    \boldsymbol A_{im}
	    &= - Re^{-1} \, \left( \nabla \bphi_m , \nabla \bphi_i \right) \label{eqn:g-rom_operators-A},
	    \\
	    \boldsymbol B_{imk}
	    &= - \bigl( \bphi_m \cdot \nabla \bphi_k , \bphi_i \bigr) 
	    \, \label{eqn:g-rom_operators-B}.   
    \end{align}

\section{Large Eddy Simulation ROM (LES-ROM)}
    \label{sec:les-rom}

The ROM closure that we investigate in this paper (i.e., the DD-VMS-ROM presented in Section~\ref{sec:dd-vms-rom}) is a large eddy simulation ROM (LES-ROM).
Thus, in this section, we briefly outline the construction of LES-ROMs.

LES-ROMs are ROM closures that have been developed over the last decade (see~\cite{wang2012proper,xie2017approximate} and the survey in Section V in~\cite{ahmed2021closures}, as well as related approaches in~\cite{girfoglio2021pod,girfoglio2021pressure}).
LES-ROMs {utilize} mathematical principles used in classical LES~\cite{BIL05,sagaut2006large} to construct ROM closure models for ROMs in under-resolved regimes, i.e., when the number of ROM basis functions is insufficient to represent the complex dynamics of the underlying flows.
Classical LES and LES-ROMs are similar in spirit:
They both aim at approximating the large scales in the flow at the available coarse resolution (e.g., coarse mesh in classical LES and not enough ROM basis functions in LES-ROMs).
Furthermore, they both use spatial filtering to define the large scales that need to be approximated.
We emphasize, however, that there are also major differences between classical LES and LES-ROMs.
One of the main differences is the type of spatial filtering used to define the large flow structures.
In classical LES, continuous filters (e.g., the Gaussian filter) are used to define the filtered equations at a continuous level.
In contrast, in LES-ROMs, due to the hierarchical structure of the ROM spaces, the ROM projection (which is a discrete spatial filter) is generally used instead.
(For a notable exception, see the ROM differential filter, which is a continuous spatial ROM filter used in~\cite{xie2017approximate} to construct the approximate deconvolution ROM closure.)
The ROM projection is used, in particular, to build variational multiscale (VMS) ROM closures (see, e.g.,~\cite{bergmann2009enablers,iliescu2013variational,iliescu2014variational,reyes2020projection,stabile2019reduced,wang2012proper} and the VMS-ROM survey in Section V.A in~\cite{ahmed2021closures}), such as the closure that we investigate in this paper, which we describe next.

To construct the DD-VMS-ROM, we start by choosing 
the ``truth" solution, i.e., the most accurate ROM solution that we can construct with the given snapshots.

\begin{definition}[Truth Solution]
    \label{def:truth-solution}
    For fixed $M$ and $d$, we define the $d$-dimensional G-ROM solution of~\eqref{eqn:d-dimensional-g-rom-full-disc} as our ``truth" solution.
\end{definition}

The goal of an LES-ROM is to construct {an $r$-dimensional ROM whose solution,
$\bu_r$, approximates as accurately as possible the large scale component of the truth solution, $P_r(\bu_d)$.} We note that, since $r \ll d$, the LES-ROM development takes place in an under-resolved regime.

In what follows, our goal is to use data to construct an LES-ROM (specifically, the DD-VMS-ROM) whose solutions are as close as possible to $P_r(\bu_d)$, i.e., the ROM projection of the truth solution.  
Thus, in the numerical analysis in Section~\ref{sec:dd-vms-rom-verifiability}, the DD-VMS-ROM solution will be compared to large scale component of the truth solution, which will be considered as \emph{data}.

{In what follows,} 
we use the LES-ROM framework to achieve the following {objectives}:
(i) Use the ROM projection to define the large ROM spatial scales;
(ii) Use the ROM projection to filter the $d$-dimensional G-ROM~\eqref{eqn:d-dimensional-g-rom-full-disc} 
{and} obtain the LES-ROM, i.e., the set of equations for the filtered ROM variables; and
(iii) Finally, use data-driven modeling to construct a ROM closure model for the filtered ROM equations {developed} in step (ii).
In this section, we discuss steps (i) and (ii); in the next section, we discuss step (iii), i.e., we construct the DD-VMS-ROM.

To define the large ROM scales and build the VMS framework, we first decompose the $d$-dimensional ROM space $\boldsymbol X^d$ into two orthogonal subspaces
 \begin{subequations} 
 \begin{align}
 \boldsymbol X^r:= \text{span} \{ \boldsymbol \varphi_1, ... , \boldsymbol \varphi_r\} \label{eqn:r-dim-set}, \\
  ( \boldsymbol X^r)^\perp := \text{span} \{ \boldsymbol \varphi_{r+1}, ... , \boldsymbol \varphi_d\} , \label{eqn:d-r-dim-set}
 \end{align}
 \end{subequations}
 where $\boldsymbol X^r$ contains the first $r$ dominant ROM basis functions, and $(\boldsymbol X^r)^ \perp$, which is orthogonal to $\boldsymbol X^r$, contains the less energetic ROM basis functions. 
We also define the following orthogonal projections:
\begin{definition}[Orthogonal Projections] \label{defn:orthogonal-projs} Let $P_r: L^2 \rightarrow \boldsymbol X^r$ be the orthogonal projection onto $\boldsymbol X^r$, and $Q_r: L^2 \rightarrow (\boldsymbol X^r)^\perp $ be the orthogonal projection onto $(\boldsymbol X^r)^\perp $, which can be defined as
\begin{subequations}
\begin{align}
 P_r (\boldsymbol u) = \sum_{i=1}^r ( \boldsymbol u, \boldsymbol \varphi_i )
 \boldsymbol \varphi_i,  \quad  \boldsymbol u \in L^2, \label{eqn:Pr_defn} \\
 Q_r (\boldsymbol u) = \sum_{i=r+1}^d ( \boldsymbol u, \boldsymbol \varphi_i )
 \boldsymbol \varphi_i,  \quad  \boldsymbol u \in L^2, \label{eqn:Qr_defn}
\end{align}
\end{subequations}
{where $L^2$ denotes the space of square integrable functions on the spatial domain.}
\end{definition}

 Next, in the LES spirit, we decompose the most accurate ROM solution at time step $n$, $\boldsymbol u_d^n$ (i.e., the $d$-dimensional G-ROM solution~\eqref{eqn:d-dimensional-g-rom-full-disc}, which is the ``truth" solution that is employed as a benchmark in our investigation) as
\begin{eqnarray}
 \boldsymbol u_d^n 
 := \underbrace{P_r(\boldsymbol u_d^n)}_{\text{large scales}} 
 + \underbrace{Q_r (\boldsymbol u_d^n)}_{\text{small scales}} ,
 \label{eqn:u-d-decomposition}
\end{eqnarray}  
where $P_r$ and $Q_r$ are the two orthogonal projections in Definition~\ref{defn:orthogonal-projs}.
Equation~\eqref{eqn:u-d-decomposition} represents the LES-ROM  decomposition of the ``truth" solution, $\boldsymbol u_d^n$, into its large scale component, $P_r(\boldsymbol u_d^n)$, and its small scale component, $Q_r(\boldsymbol u_d^n)$.  

The ROM spatial filter that we use to construct the LES-ROM is the ROM projection filter~\cite{oberai2016approximate,wang2012proper}, i.e., the orthogonal projection $P_r$ defined in Definition~\ref{defn:orthogonal-projs}, which satisfies the following equation: 
For given $ \boldsymbol u \in L^2$,
\begin{align}
\displaystyle \big( P_r(\boldsymbol u), \boldsymbol \varphi_i \big)=~ \big( \boldsymbol u, \boldsymbol \varphi_i \big),~~~~~\forall \, i=1,...,r. \label{PF} 
\end{align}

To construct the LES-ROM, we need to construct the equation satisfied by the large scales, $P_r(\boldsymbol u_d^n)$, defined in~\eqref{eqn:u-d-decomposition}.
We note that, by using Definition~\ref{defn:orthogonal-projs}
and the ROM orthogonality property, we obtain the following formula for the large scale component 
$P_r(\boldsymbol u_d^n)$:
\begin{eqnarray} \label{eqn:filtered-FOM}
   \displaystyle P_r(\boldsymbol u_d^n) =   \sum_{i=1}^{r}  
  (\boldsymbol a_d^n)_i \, \boldsymbol \varphi_i.
\end{eqnarray}
To construct the LES-ROM satisfied by $P_r(\bu_d^n)$, we apply the ROM spatial filter, $P_r$, to the equation satisfied by the ``truth'' solution, $\bu_d^n$ (i.e., to the full discretization of the $d$-dimensional G-ROM~\eqref{eqn:d-dimensional-g-rom-full-disc}), we restrict the test functions in~\eqref{eqn:d-dimensional-g-rom-full-disc} to the $r$-dimensional ROM subspace $\bX^r$ defined in \eqref{eqn:r-dim-set}, and we use the decomposition~\eqref{eqn:u-d-decomposition}.
This yields the equations satisfied by the large scales, $P_r(\bu_d^n)$, i.e., the LES-ROM equations: 
\begin{eqnarray} \label{eqn:filtered-weak-form} 
\begin{aligned}
     \left( \frac{P_r(\boldsymbol u_d^n)-P_r(\boldsymbol u_d^{n-1}) }{\Delta t}, \boldsymbol v_r \right) 
     &+ Re^{-1} (\nabla P_r(\boldsymbol u_d^n),  \nabla \boldsymbol v_r )
     + (P_r(\boldsymbol u_d^n)\cdot\nabla P_r(\boldsymbol u_d^n), \boldsymbol v_r ) 
     \\
    &+ \mathcal{E}^n + 
    ( \boldsymbol \tau^{FOM}( \boldsymbol u_d^n) , \boldsymbol v_r ) 
    = 
    (\boldsymbol f^n , \boldsymbol v_r ) \, , \qquad \forall \, \boldsymbol v_r \in \boldsymbol X^r,
     \end{aligned}
 \end{eqnarray} 
where we used that, by~\eqref{PF},  
$(P_r(\boldsymbol f^n), \boldsymbol v_r ) = ( \boldsymbol f^n , \boldsymbol v_r )$.
In the LES-ROM equations~\eqref{eqn:filtered-weak-form},
 the Reynolds stress tensor $\boldsymbol \tau^{FOM}(\boldsymbol u_d^n)$ and commutation error $\mathcal{E}$ are defined as follows:
 \begin{eqnarray}
  \boldsymbol \tau^{FOM} (\boldsymbol u_d^n) := \boldsymbol u_d^n \cdot\nabla \boldsymbol u_d^n  - P_r(\boldsymbol u_d^n) \cdot\nabla  P_r(\boldsymbol u_d^n), 
  \label{eqn:tau} \\
  \mathcal{E}^n: = Re^{-1} (\nabla Q_r (\boldsymbol u_d^n), \nabla \boldsymbol v_r), \label{eqn:ce}
 \end{eqnarray}
 respectively. 
 We note that, to obtain the LES-ROM equations~\eqref{eqn:filtered-weak-form}, we used the fact that the term $(Q_r (\boldsymbol u_d^n), \boldsymbol v_r) $ vanishes since $Q_r (\boldsymbol u_d^n)$ is orthogonal to  any vector in $\boldsymbol X^r$. 
 We also note that the term $(\nabla Q_r(\boldsymbol u_d^n), \nabla \boldsymbol v_r)$ in the commutation error term~\eqref{eqn:ce}  does not vanish since the ROM basis functions are only $L^2$-orthogonal, not $H_0^1$-orthogonal.

 \begin{remark}[Commutation Error]
 In \cite{koc2019commutation}, we investigated the effect of the commutation error~\eqref{eqn:ce} on ROMs.  
 We showed that the commutation error is generally nonzero, but becomes negligible for large $Re$. 
 Since our current investigation centers around LES-ROMs for turbulent flows, for simplicity, 
 we do not consider the commutation error.
\end{remark}

\begin{definition}[Closure Model]
A closure model consists of replacing 
the Reynolds stress tensor $\boldsymbol \tau^{FOM} (\boldsymbol u_d^n)$ in \eqref{eqn:filtered-weak-form} 
with another tensor 
$\boldsymbol \tau^{ROM} (P_r(\boldsymbol u_d^n)) $
depending only on 
$P_r(\boldsymbol u_d^n)$. 
\end{definition} 

Thus, the role of the closure model $\boldsymbol \tau^{ROM}$ is to replace the true closure model $\boldsymbol \tau^{FOM} (\boldsymbol u_d^n)$ (which cannot be computed in $\bX^r$) with a term that can actually be computed in $\bX^r$.
Since a closure model cannot in general be exact (i.e., $\boldsymbol \tau^{FOM} (\boldsymbol u_d^n )\neq \boldsymbol \tau^{ROM} (P_r(\boldsymbol u_d^n))  $),
when $\boldsymbol \tau^{ROM} (P_r(\boldsymbol u_d^n))$ is inserted for $\boldsymbol \tau^{FOM} (\boldsymbol u_d^n) $ in \eqref{eqn:filtered-weak-form} the solution of the resulting system is just an approximation to $P_r(\boldsymbol u_d^n)$. 
We denote this LES-ROM approximation to $P_r(\boldsymbol u_d^n)$ as $\boldsymbol u_r^n$, which can be written as
 \begin{eqnarray} \label{eqn:dd-vms-soln}
   \displaystyle \boldsymbol u_r^n =  \sum_{i=1}^{r} ( \boldsymbol a_r^n)_i \, \boldsymbol \varphi_i.
 \end{eqnarray}
Thus, the LES-ROM equations for $\bu_r^n$ are
\begin{eqnarray}\label{eqn:r-dimensional-dd-vms-rom} 
\begin{aligned}
    \hspace*{-0.8cm} 
     \left( \frac{\boldsymbol u_r^n -\boldsymbol u_r^{n-1}}{\Delta t}, \boldsymbol v_r \right) 
     + Re^{-1} (\nabla \boldsymbol  u_r^n, \nabla \boldsymbol v_r )
    + (\boldsymbol u_r^n \cdot \nabla \boldsymbol u_r^n, \boldsymbol v_r ) 
     + ( \boldsymbol \tau^{ROM} (\boldsymbol u_r^n), \boldsymbol v_r ) 
     = (\boldsymbol f^n, \boldsymbol v_r ), \  \forall \, \boldsymbol v_r \in \boldsymbol X^r.
     \end{aligned}
 \end{eqnarray} 
Inserting \eqref{eqn:dd-vms-soln} into \eqref{eqn:r-dimensional-dd-vms-rom} yields the following matrix form of the LES-ROM:

\begin{eqnarray} 
	\frac{\boldsymbol a_r^n-\boldsymbol a_r^{n-1}}{\Delta t}
	=
	\boldsymbol b^n + \boldsymbol A \boldsymbol a_r^n + (\boldsymbol a_r^n)^T \boldsymbol B \boldsymbol a_r^n 
	+ 
	[-( \boldsymbol \tau^{ROM} (\boldsymbol u_r^n), \boldsymbol \varphi_{i} )_{i=1, \ldots, r}],
	\label{eqn:vms-rom} 
\end{eqnarray}
where the vector $\boldsymbol b^n$, the matrix $\boldsymbol A$, and the tensor $\boldsymbol B$ are defined in \eqref{eqn:g-rom_operators-b}-\eqref{eqn:g-rom_operators-B}, {but {here are} truncated to the first $r$ components, 
{ i.e.,} the indices $i,k,m$ in \eqref{eqn:g-rom_operators-b}-\eqref{eqn:g-rom_operators-B} are restricted between $1$ and $r$. 
We opt for this slight abuse of notation 
{ in order to avoid}
introducing new variables 
{that would overload the presentation.  We also note} 
that $[-( \boldsymbol \tau^{ROM} (\boldsymbol u_r^n), \boldsymbol \varphi_{i} )_{i=1, \ldots, r}]$ in 
{ \eqref{eqn:vms-rom} }
denotes the $r\times 1$ vector whose $i^{\mathrm{th}}$ component is given by $-( \boldsymbol \tau^{ROM} (\boldsymbol u_r^n), \boldsymbol \varphi_{i})$}.

\section{Data Driven Variational Multiscale ROM (DD-VMS-ROM)} 
\label{sec:dd-vms-rom}

In this section, we outline the construction of the data-driven variational multiscale ROM (DD-VMS-ROM) closure model proposed in~\cite{mou2021data,xie2018data}. 
We also describe the physical constraints that we add to the  DD-VMS-ROM in order to increase its stability and accuracy.
The construction of the DD-VMS-ROM is carried out within the LES-ROM framework described in Section~\ref{sec:les-rom}.

To construct the DD-VMS-ROM, we start from the LES-ROM equations~\eqref{eqn:vms-rom}.
First, we notice that since we used the ROM projection as a spatial filter, the LES-ROM~\eqref{eqn:vms-rom} is in fact a variational multiscale ROM (VMS-ROM).
However, the VMS-ROM~\eqref{eqn:vms-rom} is not closed since the closure term $\btau^{ROM}(\bu_r^n)$ still needs to be determined.
To construct a VMS-ROM closure model, we use data-driven modeling.
Specifically, we first postulate a linear ansatz for the VMS-ROM closure term, and then we determine the parameters in the linear ansatz that best match the FOM data.
The linear ansatz for the VMS-ROM closure term can be written as follows:
\begin{eqnarray} \label{eqn:ansatz}
\begin{aligned}
[- ( \boldsymbol \tau^{ROM} (\boldsymbol u_r^n), \boldsymbol \varphi_{i} )_{i=1, \ldots, r}] 
 \approx \widetilde{\boldsymbol A} \, \boldsymbol a_r^n ,
\end{aligned}
\end{eqnarray}
where $\boldsymbol a_r^n$ is the vector of ROM coefficients of the solution $ \boldsymbol u_r^n$; {cf.~\eqref{eqn:dd-vms-soln}}. 
To determine the $r \times r$ matrix $\widetilde{\boldsymbol A}$ in \eqref{eqn:ansatz}, in the offline  stage, we solve the following low-dimensional {\it least squares problem}:

\begin{eqnarray} \label{eqn:least-squares}
\begin{aligned}
	\min_{\widetilde{\boldsymbol A}} \ \sum_{n=1}^{M} 
	\biggl\| 
	- \biggl[ 
		\biggl(\boldsymbol u_d^n \cdot\nabla\boldsymbol u_d^n
		& -  
		P_r(\boldsymbol u_d^n)\cdot\nabla P_r(\boldsymbol u_d^n)\, , \, \boldsymbol \varphi_{i} \biggr)_{i=1, \ldots, r} 
		\biggl]  - \underbrace{ [\bigl( \widetilde{\boldsymbol A} \, {\bb_r^n} \bigr)_{i=1,..,r}] }_{:= [({  \boldsymbol \tau^{ROM}  (P_r({\bu_d^n}))}, \, \boldsymbol \varphi_{i} )_{i=1, \ldots, r}]}
	\biggr\|^2 \, ,
\end{aligned}
\end{eqnarray}
where $\boldsymbol u_d^n$ and $P_r(\boldsymbol u_d^n)$ are obtained from the available FOM data and are defined in \eqref{eqn:ud-representation} and \eqref{eqn:filtered-FOM}, respectively, 
{and 
$\bb_r^n$ is the $r$-dimensional vector that contains the first $r$ entries of the vector $\boldsymbol a_d^n$.
}

\paragraph{Physical Constraint}
In the numerical investigation in~\cite{CSB03}, it was shown that, in the mean, the LES-ROM closure model dissipates energy. 
Thus, to mimic this behavior, in~\cite{mohebujjaman2019physically} we equipped the DD-VMS-ROM with a similar physical constraint.
Specifically, in the least squares problem \eqref{eqn:least-squares}, we added the constraint that $\widetilde{\boldsymbol A}$ be negative semidefinite: 
\begin{eqnarray} \label{eqn:physical-constraint}
 (\boldsymbol a_r^n)^T \widetilde{\boldsymbol A} \boldsymbol a_r^n \leq 0
 \qquad
 \forall \, \ba_r^n \in \R^r.
\end{eqnarray}
{For the numerical results presented in Section~\ref{sec:num}, this condition \eqref{eqn:physical-constraint} is guaranteed by enforcing a particular structure on $\widetilde{\boldsymbol A}$. 
Specifically, we
require the entries of $\widetilde{\boldsymbol A}$ to satisfy 
the following relations:} 
\begin{equation} \label{Eq_phys_constraint}
{\widetilde{A}_{ij} = - \widetilde{A}_{ji}, \; \forall \; i \neq j, \text{ and } \widetilde{A}_{ii} \le 0, \; \forall \; i.}
\end{equation}

Solving the least squares problem~\eqref{eqn:least-squares} with the physical constraint~\eqref{eqn:physical-constraint},  using the resulting matrix $\widetilde{\boldsymbol A}$ in the linear ansatz~\eqref{eqn:ansatz}, and plugging this in the VMS-ROM~\eqref{eqn:vms-rom} 
yields the data-driven variational multiscale ROM (DD-VMS-ROM):
\begin{eqnarray} 
	\frac{\boldsymbol a_r^n-\boldsymbol a_r^{n-1}}{\Delta t}
	= \boldsymbol b^n +
	( \boldsymbol A + \widetilde{\boldsymbol A}) \boldsymbol a_r^n + (\boldsymbol a_r^n)^T \boldsymbol B \boldsymbol a_r^n.
	\label{eqn:dd-vms-rom} 
\end{eqnarray}

\section{Verifiability of the DD-VMS-ROM}
    \label{sec:dd-vms-rom-verifiability}

In this section, we prove the verifiability of the DD-VMS-ROM described in Section~\ref{sec:dd-vms-rom}.
In Section~\ref{sec:verifibiality-limit-consistency-conditions}, we introduce the verifiability and mean dissipativity concepts in the ROM setting.
In Section~\ref{sec:dd-vms-rom-verifiability-proof}, we prove that the DD-VMS-ROM is verifiable.

\subsection{Definition of Verifiability and Mean Dissipativity} \label{sec:verifibiality-limit-consistency-conditions}
The goal of this subsection is to define the verfiability 
of ROM closure models.
Verifiability of closure models has been investigated for decades in classical CFD (see, e.g.,~\cite{kaya2002verifiability} as well as~\cite{BIL05} for a survey of verifiability methods in LES).
We emphasize, however, that, to our knowledge, the verifiability concept has not been defined in a ROM context.
In this section, we take a first step in this direction and define verifiability of ROM closure models.
We also define the mean dissipativity of ROM closures, which will be used in Section~\ref{sec:dd-vms-rom-verifiability-proof} to prove the verifiability of the DD-VMS-ROM.

{
In the remainder of this paper, we also use the following notation: 

\begin{definition}[Generic Constant $C$]         \label{def:constant}
We denote with $C$ a generic constant that can depend on the fixed data (e.g., the solution, $\bu$, the number of snapshots, $M$, the number of linearly independent snapshots, $d$, and the ``truth'' solution, $\bu_d$), but not on the ROM  parameters (e.g., the ROM dimension, $r$, and the ROM solution, $\bu_r$).
\end{definition}
}

\begin{definition}[Verifiability] \label{definition:DD-VMS-ROM-verifiability}

Let the number of snapshots, $M$, (and, thus, the number of linearly independent snapshots, $d$) be fixed.
A ROM closure model is verifiable in the $L^2$ norm,  
{$|| \cdot  ||, $} 
if there is a constant $C$ such that, for all $r \leq d$ and for all $n = 1, \ldots, M$, the following {\it a priori} error bound holds: 

\begin{eqnarray} \label{eqn:DD-VMS-ROM-verifiability}
 \boxed{ 
|| \,P_r(\bu_d^n) - \boldsymbol u_r^n ||^2
\leq 
 C \frac{1}{n} \sum_{j=1}^{n} 
  || P_r(\,\boldsymbol \tau^{FOM}( \boldsymbol u_d^j) - \boldsymbol \tau^{ROM} (P_r(\bu_d^j)\,)||^2 , 
 }
 \end{eqnarray}
 where $\boldsymbol u_d^j$ 
 represents the ``truth" solution (i.e., the $d$-dimensional G-ROM solution of~\eqref{eqn:d-dimensional-g-rom-full-disc}) at $t = t_j, \ j = 1, \ldots, M$, 
 and $\boldsymbol u_r^n$ solves the 
 ROM equipped with the given ROM closure model
 at $t = t_n, \ n = 1, \ldots, M$.
\end{definition}

Definition~\ref{definition:DD-VMS-ROM-verifiability} says that a ROM closure model is verifiable if a small average error in the ROM closure term implies a small error in the LES-ROM approximation.

{
\begin{remark}[\textit{A Priori} Error Bound]
We emphasize that inequality~\eqref{eqn:DD-VMS-ROM-verifiability} in the verifiability definition is an \textit{a priori} error bound.
This ROM error bound is similar to the \textit{a priori} error bounds for classical FOMs, e.g., the FE method, which are often of the following form (see, e.g., Theorem 1.5 in~\cite{thomee2006galerkin}):
\begin{eqnarray}
    \text{error}
    \leq C \left( h^{p_1} 
                + \Delta t^{p_2} \right), 
    \label{eqn:remark-a-priori-1}
\end{eqnarray}
where $h$ is the spatial mesh size, $\Delta t$ is the time step, $p_1$ and $p_2$ are exponents that depend on the particular finite element and time discretization used, and $C$ is a generic constant that can depend on the problem data (including the solution of the continuous problem), but not on the discretization parameters.
As explained in Section 2.4 of~\cite{layton2008introduction}, the \textit{a priori} error bound~\eqref{eqn:remark-a-priori-1} shows asymptotic convergence as $h \rightarrow 0$ and $\Delta t \rightarrow 0$, and can give the asymptotic rate of convergence with respect to the spatial and temporal discretizations. 
We emphasize that one essential feature of the FE \textit{a priori} error bound~\eqref{eqn:remark-a-priori-1} is that it can be proven \textit{before} actually running the FE model (which explains the error bound's \textit{a priori} qualifier).
We note, however, that since the constant $C$ on the right-hand side of~\eqref{eqn:remark-a-priori-1} can depend on the unknown solution of the continuous problem, the \textit{a priori} error bound~\eqref{eqn:remark-a-priori-1} cannot be used to decide where the spatial mesh should be refined or coarsened.
For that purpose, one could instead use \textit{a posteriori} error bounds, in which the right-hand side depends entirely on computable quantities, e.g., the FE solution~\cite{ainsworth2000posteriori}. 

The ROM error bound~\eqref{eqn:DD-VMS-ROM-verifiability} in the verifiability definition is similar to the \textit{a priori} FE error bound~\eqref{eqn:remark-a-priori-1}. 
Indeed, the right-hand side of~\eqref{eqn:DD-VMS-ROM-verifiability} does not depend on the ROM solution and can be evaluated \textit{before} actually running the ROM.
Thus, the ROM error bound~\eqref{eqn:DD-VMS-ROM-verifiability} is an \textit{a priori} error bound, just like the FE error bound~\eqref{eqn:remark-a-priori-1}. 
Furthermore, the right-hand side of~\eqref{eqn:DD-VMS-ROM-verifiability} is the product of a generic constant that does not depend on the ROM discretization parameters, and a term that can be tuned by the user (i.e., the average ROM closure error term). 
Thus, as the average ROM closure error in~\eqref{eqn:DD-VMS-ROM-verifiability} decreases, we expect the ROM error to decrease at the same rate.
Our numerical investigation in Section~\ref{sec:num} shows that this is indeed the case.
There is, however, a difference between the \textit{a priori} ROM error bound~\eqref{eqn:DD-VMS-ROM-verifiability} and the \textit{a priori} FE error bound~\eqref{eqn:remark-a-priori-1}:
The latter depends on two FE parameters that can be easily adjusted (i.e., the spatial mesh size, $h$, and the time step, $\Delta t$).
The former, however, depends on the average ROM closure error, which can be tuned by varying the  parameters in the numerical discretization of the least squares problem~\eqref{eqn:least-squares}.
This process is explained in Sections~\ref{sec:numerical-implementation} and \ref{sec:assessment-results}.

\label{remark:a-priori}
\end{remark}
}

{
\begin{remark}
We note that the terms on the right-hand side of~\eqref{eqn:DD-VMS-ROM-verifiability} in the verifiability definition are the same as those used in the least squares problem~\eqref{eqn:least-squares}.
Furthermore, the $L^2$ norm is used in both \eqref{eqn:DD-VMS-ROM-verifiability} and \eqref{eqn:least-squares}.
Thus, solving the least squares problem~\eqref{eqn:least-squares} to construct the DD-VMS-ROM and proving that the DD-VMS-ROM is verifiable (as we will do in Theorem~\eqref{theorem:main-theorem}) should yield accurate DD-VMS-ROM approximations.
The numerical investigation in Section~\ref{sec:num} will show that, as expected, the DD-VMS-ROM approximations are accurate.
\label{remark:l2-norm}
\end{remark}
}

\begin{definition}[Mean Dissipativity]
\label{definition:mean-dissipativity}
A ROM closure model satisfies the mean dissipativity condition if
{for {the} $\bu_d^n, \bu_r^n$, and $n$ 
given in Definition~\ref{definition:DD-VMS-ROM-verifiability}, the following inequalities are satisfied}:
\begin{eqnarray} \label{eqn:mean-dissipativity-partb}
\boxed{ 
0 \leq ( 
 \boldsymbol \tau^{ROM}(P_r(\bu_d^n)) - \boldsymbol \tau^{ROM} (\boldsymbol u_r^n ) 
\, , \, 
P_r(\bu_d^n) - \boldsymbol u_r^n ) < \infty .
}
\end{eqnarray}
\end{definition}

\subsection{Proof of DD-VMS-ROM's Verifiability}
    \label{sec:dd-vms-rom-verifiability-proof}

In this section, we first prove that the DD-VMS-ROM is mean dissipative.
Then, we use this result to prove that the DD-VMS-ROM is verifiable.

\begin{theorem} \label{theorem:mean-dissipativity}
The DD-VMS-ROM with linear ansatz \eqref{eqn:dd-vms-rom} and physical constraint \eqref{eqn:physical-constraint} satisfies mean dissipativity according to Definition~\ref{definition:mean-dissipativity}.
\end{theorem}

\begin{proof}

The least squares problem \eqref{eqn:least-squares} yields the ROM operator $\widetilde{\boldsymbol A}$ for $-( \boldsymbol \tau^{ROM}(P_r(\boldsymbol u_d^n) , \boldsymbol \varphi_i)$, which is the VMS-ROM closure term. 
We note that the same ROM operator $\widetilde{\boldsymbol A}$ is used to construct the VMS-ROM closure term
$-( \tau^{ROM}(\boldsymbol u_r^n) , \boldsymbol \varphi_i)$. 
Specifically, the ROM operator $\widetilde{\boldsymbol A}$ that is created by solving the least squares problem~\eqref{eqn:least-squares} for the VMS-ROM closure term 
$-( \boldsymbol \tau^{ROM}(P_r(\boldsymbol u_d^n) , \boldsymbol \varphi_i)$
is used in the linear ansatz 
$
-(\boldsymbol \tau^{ROM}(P_r(\boldsymbol u_d^n) , \boldsymbol \varphi_i)_{i=1, \ldots, r} \approx \widetilde{\boldsymbol A} \, \boldsymbol b_r$, where $\boldsymbol b_r^n$ is 
{the $r$-dimensional vector defined in~\eqref{eqn:least-squares}, i.e.,} 
{the} $r$-dimensional vector that contains the first $r$ entries of the vector $\boldsymbol a_d^n$.
The same ROM operator $\widetilde{\boldsymbol A}$ is also used in the linear ansatz~\eqref{eqn:ansatz} for the VMS-ROM closure term 
$-( \boldsymbol \tau^{ROM}(\boldsymbol u_r^n) , \boldsymbol \varphi_i)$:
$-(  \boldsymbol \tau^{ROM}(\boldsymbol u_r^n) , \boldsymbol \varphi_i)_{i=1, \ldots, r} \approx \widetilde{\boldsymbol A} \, \boldsymbol a_r$.
We approximate the VMS-ROM closure terms with these ansatzes and 
obtain the following equalities:
\begin{eqnarray} \label{eqn:proof-part1}
\begin{aligned}
 ( 
   \boldsymbol \tau^{ROM} (P_r(\boldsymbol u_d^n))  - \boldsymbol \tau^{ROM} (\boldsymbol u_r^n )
 \, , \, \boldsymbol \varphi_i )
 & =  \Big( \boldsymbol \tau^{ROM} (P_r(\boldsymbol u_d^n))  \, , \, \boldsymbol \varphi_i \Big) - \Big(  \boldsymbol \tau^{ROM} (\boldsymbol u_r^n ) \, , \, \boldsymbol \varphi_i \Big)  \\
 & = (-\widetilde{\boldsymbol A} \, \boldsymbol b_r^n )_i - (-\widetilde{\boldsymbol A} \, \boldsymbol a_r^n)_i \\
& = \big( -\widetilde{\boldsymbol A} \, (\boldsymbol b_r^n - \boldsymbol a_r^n ) \big)_i  \qquad \forall i=1,..,r .
\end{aligned}
\end{eqnarray}
To prove that the inner product $(   \boldsymbol \tau^{ROM} (P_r(\boldsymbol u_d^n))  - \boldsymbol \tau^{ROM} (\boldsymbol u_r^n ) \, , \,  P_r(\boldsymbol u_d^n)- \boldsymbol u_r^n )$ is non-negative, we use the definitions of $P_r(\boldsymbol u_d^n)$ in \eqref{eqn:filtered-FOM} and $\boldsymbol u_r^n$ in \eqref{eqn:dd-vms-soln},  and rewrite it as follows:
\begin{eqnarray} \label{eqn:proof-part2}
 \begin{aligned}
\Big(\boldsymbol \tau^{ROM} (P_r(\boldsymbol u_d^n))  - \boldsymbol \tau^{ROM} (\boldsymbol u_r^n ) \, , \,  P_r(\boldsymbol u_d^n)- \boldsymbol u_r^n \Big)
& =  \Big( \boldsymbol \tau^{ROM} (P_r(\boldsymbol u_d^n))  - \boldsymbol \tau^{ROM} (\boldsymbol u_r^n )  \, , \,  \sum_{i=1}^{r} (\boldsymbol a_d^n - \boldsymbol a_r^n )_i \, \boldsymbol \varphi_i \, \Big) \\
& = \sum_{i=1}^{r} (\boldsymbol a_d^n - \boldsymbol a_r^n)_i \,  \Big(  \boldsymbol \tau^{ROM} (P_r(\boldsymbol u_d^n))  - \boldsymbol \tau^{ROM} (\boldsymbol u_r^n )  \, , \, \, \boldsymbol \varphi_i \, \Big).
 \end{aligned}
\end{eqnarray}
By applying \eqref{eqn:proof-part1} to \eqref{eqn:proof-part2} and using the physical constraint \eqref{eqn:physical-constraint}, we get
\begin{eqnarray}
 \begin{aligned}
  (\boldsymbol \tau^{ROM} (P_r(\boldsymbol u_d^n))  - \boldsymbol \tau^{ROM} (\boldsymbol u_r^n )  \, , \,  P_r(\boldsymbol u_d^n)- \boldsymbol u_r^n )
& = \sum_{i=1}^{r} (\boldsymbol a_d^n - \boldsymbol a_r^n)_i \,  \big(-\widetilde{\boldsymbol A} \, (\boldsymbol b_r^n - \boldsymbol a_r^n ) \big)_i \\
& = - (\boldsymbol b_r^n - \boldsymbol a_r^n )^T \,  \widetilde{\boldsymbol A} \, (\boldsymbol b_r^n - \boldsymbol a_r^n ) \geq 0 ,
 \end{aligned}
 \label{eqn:proof-2}
\end{eqnarray}
since $\widetilde{\boldsymbol A}$ is negative semi-definite. 
In~\eqref{eqn:proof-2}, we have used that $\boldsymbol b_r^n$ is an $r$-dimensional vector that contains the first $r$ entries of the $\boldsymbol a_d^n$.
The inequality in~\eqref{eqn:proof-2} concludes the proof.
\end{proof}

\begin{remark}
We note that in Theorem~\ref{theorem:mean-dissipativity} we proved the ROM mean dissipativity property only for  $P_r(\boldsymbol u_d^n)$ and $\boldsymbol u_r^n$. 
This is contrast with the FEM context, where mean dissipativity is proven for general FEM functions (see, e.g., \cite{kaya2002verifiability}). {
{However, }
the result presented in Theorem~\ref{theorem:mean-dissipativity} is sufficient for 
{proving }
the verifibility property given in Theorem~\ref{theorem:main-theorem} below.}
\end{remark}

Next, we prove that the DD-VMS-ROM is verifiable.
We note that, as explained in Section~\ref{sec:les-rom}, 
the goal for the DD-VMS-ROM solution is to approximate as accurately as possible $P_r(\bu_d^n)$, which is the large scale component of the $d$-dimensional G-ROM solution~\eqref{eqn:d-dimensional-g-rom-full-disc}, 
{i.e.,} the ``truth'' solution that is employed as a benchmark in our investigation.
{Furthermore, as explained in the second paragraph following Definition~\ref{def:truth-solution}, the ``truth'' solution, $\bu_d$, will be considered as given data.}
We also note that $P_r(\bu_d^n)$ satisfies the LES-ROM equations~\eqref{eqn:filtered-weak-form}, which, for clarity, we rewrite below:
\begin{eqnarray} \label{eqn:filtered-weak-form-full-disc}
\begin{aligned} (\frac{P_r(\bu_d^n)-P_r(\bu_d^{n-1})}{\Delta t}, \boldsymbol v_r  ) 
     + Re^{-1} (\nabla P_r(\bu_d^n),  \nabla \boldsymbol v_r )
     + (P_r(\bu_d^n)\cdot\nabla  P_r(\bu_d^n), \boldsymbol v_r  ) 
     \\
    +  (\boldsymbol \tau^{FOM}( \boldsymbol u_d^n) , \boldsymbol v_r  ) = 
    ( \boldsymbol f^n , \boldsymbol v_r  ) , 
     \end{aligned}
 \end{eqnarray} 
where we used the fact that ($\boldsymbol \tau^{FOM}( \boldsymbol u_d^n) , \boldsymbol v_r  )$ is equal to $(P_r(\boldsymbol \tau^{FOM}( \boldsymbol u_d^n)) , \boldsymbol v_r$).
We also rewrite the full discretization of the DD-VMS-ROM~\eqref{eqn:r-dimensional-dd-vms-rom}: 
\begin{eqnarray}\label{eqn:r-dimensional-dd-vms-rom-full-disc} 
\begin{aligned}
     ( \frac{\boldsymbol u_r^n -\boldsymbol u_r^{n-1} }{\Delta t}, \boldsymbol v_r  ) 
     + Re^{-1} (\nabla \boldsymbol  u_r^n, \nabla \boldsymbol v_r  )
    + (\boldsymbol u_r^n \cdot\nabla  \boldsymbol u_r^n, \boldsymbol v_r  ) &
     \\
     + (  \boldsymbol \tau^{ROM} (\boldsymbol u_r^n), \boldsymbol v_r  )     
     & = (\boldsymbol f^n, \boldsymbol v_r  ).
     \end{aligned}
 \end{eqnarray} 
Furthermore, we use the linear ansatz~\eqref{eqn:ansatz} and the physical 
constraint~\eqref{eqn:physical-constraint} 
 for the ROM closure model in the DD-VMS-ROM~\eqref{eqn:r-dimensional-dd-vms-rom-full-disc}.
 We also choose the initial condition $\boldsymbol u_r^0 = P_r(\boldsymbol u_d^0)$.
 
The DD-VMS-ROM error at time step $n$, which we denote with $\boldsymbol e^n$, is defined as the difference between the large scale component of the ``truth'' solution, $P_r(\bu_d^n)$ (which is the solution of~\eqref{eqn:filtered-weak-form-full-disc}), and the DD-VMS-ROM solution of~\eqref{eqn:r-dimensional-dd-vms-rom-full-disc}, $\boldsymbol u_r^n$: 
$\boldsymbol e^n = P_r(\bu_d^n) - \boldsymbol u_r^n$.

To prove the DD-VMS-ROM's verifiability, we use the following 
bound on the nonlinear term, which is given in Lemma 22 in~\cite{layton2008introduction} (see also Lemma 61.1 in~\cite{sell2013dynamics}):

\begin{lemma}
Let $\Omega \subset \mathbbm{R}^{q}$ be an 
open, bounded set of class $C^2$, with $q=2$ or $3$.
For all $\boldsymbol u, \boldsymbol v, \boldsymbol w \in [\bH_0^1(\Omega)]^{q}$,
\begin{eqnarray}
b(\boldsymbol u, \boldsymbol v, \boldsymbol w) \leq C(\Omega) \sqrt{||\boldsymbol u|| \, ||\nabla \boldsymbol u||} \, || \nabla \boldsymbol v|| \, ||\nabla \boldsymbol w||,
\end{eqnarray}
where 
the trilinear form $b(\cdot,\cdot,\cdot)$~\cite{layton2008introduction,temam2001navier} is defined as
\begin{eqnarray}
b(\bu,\boldsymbol v,\bw) = (\bu \cdot\nabla \boldsymbol v, \bw ). 
\end{eqnarray}
\label{lemma:trilinear-bound}
\end{lemma}

\begin{theorem} \label{theorem:main-theorem}
The DD-VMS-ROM~\eqref{eqn:r-dimensional-dd-vms-rom-full-disc} with linear ansatz \eqref{eqn:ansatz}, physical constraint \eqref{eqn:physical-constraint}, and the initial condition $\bu_r^0=P_r(\bu_d^0)$
is verifiable: 
For a small enough time step, $\Delta t \, d_j < 1, \, \, \forall \, j= 1,...,M$, where $d_j =\Big( {\frac{27 (Re)^3 C(\Omega)^4}{16}} ||\nabla P_r( \boldsymbol u_d^j)||^4 +Re \Big)$ and 
$C(\Omega)$
is the constant in Lemma~\ref{lemma:trilinear-bound}, 
the following inequality holds 
for all 
$ n = 1, \ldots, M$: 
\begin{eqnarray} \label{eqn:main-theorem-parta}
\begin{aligned}
||\boldsymbol e^{n}||^2 + \Delta t \sum_{j=1}^{n} Re^{-1} || \nabla \boldsymbol e^{j}||^2 \leq \\
\exp \Big(\Delta t \sum_{j=1}^{n} \frac{d_j}{1-\Delta t d_j} \Big) \Big(\Delta t  \sum_{j=1}^{n} Re^{-1} 
|| P_r(\boldsymbol \tau^{FOM} (\boldsymbol u_d^j )- & \boldsymbol \tau^{ROM} (P_r(\bu_d^j))) \, ||^2\Big), \\ 
\end{aligned}
\end{eqnarray}
{where $\boldsymbol e^n = P_r(\bu_d^n) - \boldsymbol u_r^n$.}
\end{theorem}

\begin{proof}
We subtract \eqref{eqn:r-dimensional-dd-vms-rom-full-disc} from \eqref{eqn:filtered-weak-form-full-disc}, and replace $n$ with $j$ to get the error equation:
\begin{eqnarray} \label{eqn:verifiability-theorem-proof-part1}
\begin{aligned}
(\frac{\boldsymbol e^{j}-\boldsymbol e^{j-1}}{\Delta t}, \boldsymbol v_r ) &+ Re^{-1} (\nabla \boldsymbol e^{j}, \nabla \boldsymbol v_r  ) + b(P_r(\bu_d^{j}),P_r(\bu_d^{j}), \boldsymbol v_r ) - b(\boldsymbol u_r^{j},\boldsymbol u_r^{j},\boldsymbol v_r )   \\
&  + 
\big( \boldsymbol \tau^{ROM} (P_r(\bu_d^{j})) - \boldsymbol \tau^{ROM} (\boldsymbol u_r^{j}) ,   \boldsymbol v_r   \big) 
= 
- \big( \boldsymbol \tau^{FOM} (\boldsymbol u_d^{j} )- \boldsymbol \tau^{ROM} (P_r(\bu_d^{j})),  \boldsymbol v_r  \big)
. \\
\end{aligned}
\end{eqnarray} 
We set $ \boldsymbol v_r  = \boldsymbol e^{j}$
in \eqref{eqn:verifiability-theorem-proof-part1}, 
add and subtract $b(\bu_r^{j} , P_r(\bu_d^{j}), \boldsymbol e^{j} )$, and
use the fact that $b(\boldsymbol u_r^{j},\boldsymbol e^{j}, \boldsymbol e^{j} ) = 0$ 
to get the following equation:
\begin{eqnarray} \label{eqn:verifiability-theorem-proof-part2}
\begin{aligned}
\Delta t^{-1} ( \boldsymbol e^{j} - \boldsymbol e^{j-1}, \boldsymbol e^{j} ) &+ Re^{-1} || \nabla \boldsymbol e^{j} ||^2 +  b(\boldsymbol e^{j}, P_r(\bu_d^{j}),\boldsymbol e^{j})  \\
 & 
  + ( \boldsymbol \tau^{ROM} (P_r(\bu_d^{j})) - \boldsymbol \tau^{ROM} (\boldsymbol u_r^{j}) ,  \boldsymbol e^{j} )  
 = 
 - ( \boldsymbol \tau^{FOM} (\boldsymbol u_d^{j}) - \boldsymbol \tau^{ROM} (P_r(\bu_d^{j})), \boldsymbol e^{j}).
\end{aligned}
\end{eqnarray}
From Theorem~\ref{theorem:mean-dissipativity}, we have the following inequality:
\begin{eqnarray} \label{eqn:verifiability-theorem-proof-part3}
 ( \boldsymbol \tau^{ROM} (P_r(\bu_d^{j})) - \boldsymbol \tau^{ROM} (\boldsymbol u_r^{j}) ,   \boldsymbol e^{j}  ) \geq 0.
\end{eqnarray}
By applying \eqref{eqn:verifiability-theorem-proof-part3} to \eqref{eqn:verifiability-theorem-proof-part2}, we get the following inequality:
\begin{eqnarray} \label{eqn:verifiability-theorem-proof-part4}
\begin{aligned}
  \Delta t^{-1}\big(\boldsymbol e^{j}- \boldsymbol e^{j-1},  \boldsymbol e^{j}  \big) 
  & + Re^{-1} || \nabla \boldsymbol e^{j} ||^2
   \leq - b(\boldsymbol e^{j} ,P_r(\bu_d^{j}), \boldsymbol e^{j}  ) 
 - \big( \boldsymbol \tau^{FOM} (\boldsymbol u_d^{j}) - \boldsymbol \tau^{ROM} (P_r(\bu_d^{j})),  \boldsymbol e^{j}  \big).
    \end{aligned}
\end{eqnarray} 

Applying Hölder's and Young's inequalities to the  terms $( \boldsymbol e^{j} - \boldsymbol e^{j-1}, \boldsymbol e^{j} )$ and $-( \boldsymbol \tau^{FOM} (\boldsymbol u_d^{j}) - \boldsymbol \tau^{ROM} (P_r(\bu_d^{j})), \boldsymbol e^{j})$ in \eqref{eqn:verifiability-theorem-proof-part4} we obtain that, for any $C_1, C_2 > 0$, the following inequalities hold:
\begin{eqnarray} \label{eqn:verifiability-theorem-proof-part5}
\begin{aligned}
(\boldsymbol e^{j} - \boldsymbol e^{j-1} , \boldsymbol e^{j}) & =  ||\boldsymbol e^{j}||^2 - (\boldsymbol e^{j}, \boldsymbol e^{j-1}) \\
& \geq  ||\boldsymbol e^{j}||^2 - ||\boldsymbol e^{j} || \, || \boldsymbol e^{j-1}||  \\
& \geq  ||\boldsymbol e^{j}||^2 - \frac{C_1}{2} ||\boldsymbol e^{j}||^2 - \frac{1}{2C_1} ||\boldsymbol e^{j-1}||^2 
\end{aligned}
\end{eqnarray}

and 

\begin{eqnarray} \label{eqn:verifiability-theorem-proof-part7}
\begin{aligned}
|-(\boldsymbol \tau^{FOM} (\boldsymbol u_d^{j})  - \boldsymbol \tau^{ROM} (P_r(\bu_d^{j}))  ,  \boldsymbol e^{j} ) |
&= |-(P_r(\boldsymbol \tau^{FOM} (\boldsymbol u_d^{j})  - \boldsymbol \tau^{ROM} (P_r(\bu_d^{j})))  ,  \boldsymbol e^{j} )|  \\
&\leq \frac{1}{2C_2} || P_r(\boldsymbol \tau^{FOM} (\boldsymbol u_d^{j} )- \boldsymbol \tau^{ROM} (P_r(\bu_d^{j}) )) \,||^2  + \frac{C_2}{2} || \boldsymbol e^{j} ||^2.
\end{aligned}
\end{eqnarray}

  Applying Lemma~\ref{lemma:trilinear-bound} 
  to the  term $-b(\boldsymbol e^{j}, P_r(\bu_d^{j}),\boldsymbol e^{j})$,  
 we obtain the following inequality for any $C_3 > 0$:
\begin{eqnarray} \label{eqn:verifiability-theorem-proof-part6}
\begin{aligned}
|-b(\boldsymbol e^{j}, P_r(\bu_d^{j}), \boldsymbol e^{j})| 
& \leq  
C(\Omega) \, ||\nabla \boldsymbol e^{j} ||^{3/2} \, ||\nabla P_r( \boldsymbol u_d^{j}) || \, ||\boldsymbol e^{j} ||^{1/2} \\
& \leq \frac{3 C_3 C(\Omega)}{4} || \nabla \boldsymbol e^{j} ||^2 + \frac{C(\Omega)}{4{(C_3)^3}} ||\nabla P_r( \boldsymbol u_d^{j}) ||^4 || \boldsymbol e^{j} ||^2,
 \end{aligned}
\end{eqnarray}
where $C(\Omega)$ is the constant in Lemma~\ref{lemma:trilinear-bound}.

By choosing $C_1 = 1$, $C_2=Re$, and $C_3=2Re^{-1}/3 C(\Omega)$, we get the following inequality:
\begin{eqnarray} \label{eqn:verifiability-theorem-proof-part8}
\begin{aligned}
 &\frac{1}{2 \Delta t} (||\boldsymbol e^{j} ||^2 - ||\boldsymbol e^{j-1}||^2 )  + \frac{Re^{-1} }{2} || \nabla \boldsymbol e^{j}||^2  \\
 &\leq \Big( {\frac{27 (Re)^3 C(\Omega)^4}{32}} || \nabla P_r( \boldsymbol u_d^{j})||^4 + \frac{Re}{2} \Big) ||\boldsymbol e^{j} ||^2 +  \frac{Re^{-1}}{2} 
|| P_r( \boldsymbol \tau^{FOM} (\boldsymbol u_d^{j} )- \boldsymbol \tau^{ROM} (P_r(\bu_d^{j}))) \,||^2.
\end{aligned}
\end{eqnarray}

By multiplying \eqref{eqn:verifiability-theorem-proof-part8} by $2 \Delta t$ and summing the resulting inequalities from $j=1$ to 
$n$, we obtain the following inequality:
\begin{eqnarray} \label{eqn:verifiability-theorem-proof-part10}
\begin{aligned}
||\boldsymbol e^{n}||^2 + \Delta t \sum_{j=1}^{n}  Re^{-1} || \nabla \boldsymbol e^{j}||^2  \leq ||\boldsymbol e^0||^2  + \Delta t  \sum_{j=1}^{n} \Big( {\frac{27 (Re)^3 C(\Omega)^4}{16}} || \nabla P_r( \boldsymbol u_d^{j})||^4 + Re \Big) ||\boldsymbol e^{j} ||^2 \\
+ \Delta t  \sum_{j=1}^{n} Re^{-1} 
|| P_r(\boldsymbol \tau^{FOM} (\boldsymbol u_d^{j} )- \boldsymbol \tau^{ROM} (P_r(\bu_d^{j}))) \, ||^2.
\end{aligned}
\end{eqnarray}

To apply the discrete Gronwall's lemma, we first make  the following notation:
\begin{eqnarray} 
\begin{aligned} \label{eqn:discrete-gronwall-equivalence-notation}
a_{j}&:= ||\boldsymbol e^{j}||^2 \geq 0 ,   \\
b_{j}&:=  Re^{-1} || \nabla \boldsymbol e^{j}||^2  \geq 0, \\
d_{j}&:= \Big({\frac{27 (Re)^3 C(\Omega)^4}{16}} ||\nabla P_r( \boldsymbol u_d^{j})||^4 +Re \Big) \geq 0, \\
c_{j}&:= Re^{-1} || P_r( \boldsymbol \tau^{FOM} (\boldsymbol u_d^{j} )- \boldsymbol \tau^{ROM} (P_r(\bu_d^{j})) ) \, ||^2 \geq 0, \\
H&:= ||\boldsymbol e^0||^2 \geq 0. 
\end{aligned}
\end{eqnarray}
We also recall that, by the small time step assumption, the following inequality holds:
$\Delta t \, d_{j} < 1, \,  \forall j$.
By using the notation in~\eqref{eqn:discrete-gronwall-equivalence-notation}, 
we rewrite \eqref{eqn:verifiability-theorem-proof-part10} as follows:
\begin{eqnarray} \label{eqn:limit-consistency-theorem-proof-part9}
a_{n} + \Delta t \sum_{j=1}^{n} b_{j} \leq \Delta t \sum_{j=1}^{n} d_{j} \, a_{j} + \Delta t \sum_{j=1}^{n} c_{j} + H. 
\end{eqnarray}

By using the discrete Gronwall's lemma (see Lemma 27 in~\cite{layton2008introduction}) in~\eqref{eqn:limit-consistency-theorem-proof-part9}, we obtain the following inequality: 
\begin{eqnarray} \label{eqn:limit-consistency-theorem-proof-part10}
a_{n} + \Delta t \sum_{j=1}^{n} b_{j} \leq \exp \Big(\Delta t \sum_{j=1}^{n} \frac{d_{j}}{1-\Delta t d_{j}} \Big) \Big(\Delta t \sum_{j=1}^{n} c_{j} + H \Big) . 
\end{eqnarray}

We note that choosing the initial condition $\bu_r^0=P_r(\bu_d^0)$, implies that
$\boldsymbol e^0 = \bu_r^0-P_r(\bu_d^0) = 0$, and thus $H = 0$.  
As a result, \eqref{eqn:limit-consistency-theorem-proof-part10} implies that \eqref{eqn:main-theorem-parta} holds.

\end{proof}

\begin{remark}
We note that the small time step assumption that we made in the theorem, i.e., that $\Delta t \, d_{j} < 1 \, \, \forall j = 1,...,M$, is also made in a FE context (see Lemma 27 and the proof of Theorem 24 in \cite{layton2008introduction}).
\end{remark}

\begin{remark}
In this paper, we used backward Euler time discretization to obtain the full discretizations of the ROMs. 
However, other time discretization schemes could be applied as well.
\end{remark}

\section{Numerical Results}
\label{sec:num}

In Theorem~\ref{theorem:main-theorem}, we proved that the DD-VMS-ROM presented in Section~\ref{sec:dd-vms-rom} is verifiable.
In this section, we present numerical support for the theoretical results in Theorem~\ref{theorem:main-theorem}.
In Section~\ref{sec:numerical-implementation}, we provide details on the numerical implementation of the DD-VMS-ROM. 
We numerically show that the DD-VMS-ROM is verifiable for the Burgers equation in Section~\ref{sec:numerical-results-burgers} and for the flow past a cylinder in Section~\ref{sec:numerical-results-nse}.

\subsection{Numerical Implementation}
\label{sec:numerical-implementation}

\paragraph{``Truth" Solution}
For computational efficiency, instead of solving the very large-dimensional G-ROM~\eqref{eqn:d-dimensional-g-rom} to get the ``truth" solution, $\boldsymbol u_d$, we simply project the FOM data on the ROM space, i.e., $\boldsymbol u_d = P_r (\bu_h), \, r=d$.
In our numerical investigation, the two approaches yield similar results (i.e., the difference between the two approaches is on the order of the time discretization error).
Thus, using the projection of the FOM data as ``truth" solution does not affect our numerical investigation of the DD-VMS-ROM's verifiability.

\paragraph{{Least Squares Regularization: Truncated SVD}}
As is often the case in data-driven modeling~\cite{peherstorfer2016data}, the least squares problem~\eqref{eqn:least-squares} that we need to solve in order to determine the entries in the ROM closure operator $\widetilde{A}$ used to construct the DD-VMS-ROM~\eqref{eqn:dd-vms-rom} 
{can be} 
ill conditioned.
To alleviate the ill conditioning of the least squares problem, we proposed the use of the truncated SVD~\cite{xie2018data,mou2021data}
{as a regularization method~\cite[Chapter~4]{hansen2010discrete}}
(see also~\cite{yildiz2020data} for a related approach).
For completeness, in Algorithm~\ref{alg:truncated-SVD}, we outline the construction of the DD-VMS-ROM with the truncated SVD procedure.

\begin{algorithm}[htb]
	\caption{
	{Least Squares Regularization: Truncated SVD}
	}
	\label{alg:truncated-SVD}
	\begin{algorithmic}[1]
		\STATE{
		Formulate the standard linear least squares problem for the unknown vector ${\bx_u}$: 	
  \begin{align}
  \label{eqn:linear-ls-1}
	    \min_{\bx_u} \bigl\|E{\bx_u}-\boldsymbol f\bigr\|^2, 
	\end{align}
	where $E\in\mathbb{R}^{Mr\times r^2}$ is a matrix whose entries are determined by 
	$\boldsymbol a_d(t_j),j=1,\cdots, M$, $\boldsymbol f\in\mathbb{R}^{Mr\times 1}$ is a vector whose entries are determined by $P_r(\btau^{FOM}(t_j)$), and  ${\bx_u}\in\mathbb{R}^{r^2\times 1},j=1,\cdots, M$,  is a vector whose entries are determined by $ \widetilde{\boldsymbol A}$. 
					}
		\STATE{
		 Calculate the SVD of $E$: 
\begin{align}
E= U\Sigma V^\top.
\end{align}
					}
		\STATE{
		 Specify a tolerence $tol$.
					}
		\STATE{
Keep the entries in $\Sigma$ that are larger than $tol$; the resulting matrix is $\widetilde{\Sigma}$ ($\widetilde{\sigma}=\sigma$ if $\sigma>tol$; 
the singular values of $E$ can be chosen as $tol$ values).
					}
		\STATE{
Construct $\widetilde{E}$, the truncated SVD of $E$: 
\begin{align} \label{Eq_Etilde}
\widetilde{E} = \widetilde{U}\widetilde{\Sigma}\widetilde{V}^\top,
\end{align}
where $\widetilde{U}$ and $\widetilde{V}$ are the entries of $U$ and $V$ that correspond to $\widetilde{\Sigma}$, respectively.
					}					
		\STATE{
The solution is given by 
\begin{align} \label{Eq_param_without_constraint}
{\bx_u}= \left(\widetilde{V}\widetilde{\Sigma}^{-1}\widetilde{U}^\top\right){\boldsymbol f}.
\end{align}		
}
\end{algorithmic}
\end{algorithm}

The tolerance $tol$ specified in step 3 of Algorithm~\ref{alg:truncated-SVD} 
{(which yields the truncation parameter $k$, i.e., the index of the lowest singular value retained in the matrix $\widetilde{\Sigma}$ constructed in step 4 of Algorithm~\ref{alg:truncated-SVD}; see equations (4.2) and (4.3) in ~\cite[Chapter~4]{hansen2010discrete})}
plays an important role in the numerical implementation of the DD-VMS-ROM.
Specifying a large $tol$ value yields a well conditioned least squares problem in step 1 and, as a result, minimizes the numerical errors in the least squares problem.
However, a large $tol$ value also decreases the accuracy of the least squares problem, i.e., yields a DD-VMS-ROM closure operator $\tA$ that does not accurately match the FOM data.
On the other hand, choosing a small $tol$ value does not significantly decrease the accuracy of the DD-VMS-ROM closure operator $\tA$, but does not significantly alleviate the ill conditioning of the least squares problem either.
In our numerical investigation, a careful choice of the tolerance $tol$ yields optimal DD-VMS-ROM results.

{If physical constraints such as that given by \eqref{Eq_phys_constraint} are added when solving the minimization problem \eqref{eqn:linear-ls-1}, then the optimal $\widetilde{A}$ given by \eqref{Eq_param_without_constraint} associated with a specified $tol$  should be replaced by the solution 
of a constrained linear least squares solver with $\widetilde{E}$ given by \eqref{Eq_Etilde} as the data matrix. For all the numerical results presented 
in Section~\ref{sec:numerical-results-burgers} and Section~\ref{sec:numerical-results-nse}, we use the Matlab built-in solver \texttt{lsqlin} for this purpose. 
Specifically, we use the interior-point algorithm option for \texttt{lsqlin} with $\texttt{ConstraintTolerance} = 1\text{E-}10$, $\texttt{OptimalityTolerance} = 1\text{E-}9$, $\texttt{StepTolerance} = 1\text{E-}12$, and $\texttt{MaxIter} = 1000$.}

\paragraph{Time Discretization}
Although the DD-VMS-ROM's verifiability was proven in Theorem~\ref{theorem:main-theorem} for the backward Euler time discretization, in the numerical investigation 
{of the flow past a cylinder (Section~\ref{sec:numerical-results-nse}), we use the linearized BDF2 time discretization.}
We use 
this higher-order time discretizations in order to decrease the impact of the time discretization error onto the LES-ROM error, which is the main focus of the numerical investigation in this section.
Furthermore, we believe that the mathematical arguments used to prove the DD-VMS-ROM's verifiability in Theorem~\ref{theorem:main-theorem} can be extended to higher-order time discretizations such as 
that considered in 
{Section~\ref{sec:numerical-results-nse}.}

\paragraph{Criteria}
To illustrate numerically the DD-VMS-ROM verifiability proven in Theorem~\ref{theorem:main-theorem}, we use the following approach, {which was outlined in Section~\ref{sec:les-rom}
{(see, e.g., the discussion after Definition~\ref{def:truth-solution})}
and Section~\ref{sec:dd-vms-rom-verifiability} (see, e.g., Definitions~\ref{def:constant} and~\ref{definition:DD-VMS-ROM-verifiability})}:
First, we fix the number of snapshots, $M$.
Therefore, the maximal dimension of the ROM space, $d$, is also fixed.
Furthermore, the ``truth" solution $\bu_d$ (i.e., the solution of the $d$-dimensional G-ROM~\eqref{eqn:d-dimensional-g-rom}) is also fixed.
The goal of our numerical investigation is to show that, for fixed $M, d$, and $\bu_d$, there exists a constant $C$ {(see Definition~\ref{def:constant})} such that for varying $r$ values and for varying $tol$ values, the inequality~\eqref{eqn:main-theorem-parta} is satisfied.
{
Thus, the goal is to bound the error between the DD-VMS-ROM solution, $\bu_r$, and the large scale component of the ``truth'' solution, $P_r(\bu_d)$.
}

To this end, we use the following metrics:
To quantify the LES-ROM error, i.e., the {averaged error associated with the first} term on the LHS of inequality~\eqref{eqn:main-theorem-parta} {(see also the LHS of \eqref{eqn:DD-VMS-ROM-verifiability})}, we use the following 
average $L^2$ norm: 
\begin{align}
 \mathcal{E} (L^2)
 = \frac{1}{M}\sum_{n=1}^M\, 
\| P_r(\bu_d^n) - \bu_r^n \|^2
= \frac{1}{M}\sum_{n=1}^M \, \| \be^n \|^2 \, .
\label{eqn:l2-error-v}
\end{align}
To quantify the LES-ROM closure error, i.e., the term on the RHS of inequality~\eqref{eqn:main-theorem-parta}, we use the following metric:
\begin{align}
 \eta(L^2) = \frac{1}{M}\sum_{n=1}^M \left\| P_r( \btau^{FOM}(\bu_d^n) - \btau^{ROM}(P_r(\bu_d^n) \, ) \, ) \right\|^2_{L^2}.
 \label{eq:tau-cal-v}
\end{align}

{
Note that the ROM error $\mathcal{E} (L^2)$ and the closure error $\eta(L^2)$ depend on both the dimension $r$ of the LES-ROM and the aforementioned tolerance index (i.e., truncation parameter) $k$ associated with the tolerance of the truncated SVD used for constructing $\widetilde{A}$ for the given $r$. We suppressed these dependencies to simplify the notation. It should be clear from the context which parameter is varied for each of the numerical results presented below.}

\subsection{Assessment of Results}
    \label{sec:assessment-results}

To illustrate numerically the DD-VMS-ROM verifiability proven in Theorem~\ref{theorem:main-theorem}, we need to show that 
as $\eta(L^2)$ in~\eqref{eq:tau-cal-v} decreases, so does $\mathcal{E} (L^2)$ in~\eqref{eqn:l2-error-v}. {
Specifically, according to \eqref{eqn:main-theorem-parta} (see also Definition~\ref{definition:DD-VMS-ROM-verifiability}), we should see $\log(\eta(L^2))$ and $\log(\mathcal{E} (L^2))$ obey the following relation:
\begin{equation} \label{Eq_scaling_law}
\log(\mathcal{E} (L^2)) \le \alpha \log(\eta(L^2)) + \beta,
\end{equation}
with $\alpha = 1$ and some $\beta > 0$. As pointed out above, 
both $\mathcal{E} (L^2)$ and $\eta(L^2)$ depend on two parameters: 
the ROM dimension $r$ and the tolerance index $k$ in the truncated SVD. In the 
numerical investigation, we 
perform two types of experiments: 
\begin{itemize}
\item[(i)] For a fixed $r$, we aim to show that \eqref{Eq_scaling_law} holds with $\alpha \ge 1$ as $k$ is varied; 
\item[(ii)] For each $r$, we pick 
the corresponding $k$ that minimizes $\mathcal{E} (L^2)$, and 
aim to show that \eqref{Eq_scaling_law} holds with $\alpha \ge 1$ as $r$ is varied. 
\end{itemize}
Since in practice one is interested in the settings 
for which $\eta(L^2)$ is relatively small, 
a rate $\alpha>1$ indicates a better rate than the rate predicted by 
Theorem~\ref{theorem:main-theorem}.} 

We {would 
like to} note that our numerical investigation is somewhat different from the standard investigations used in the numerical analysis literature. {While increasing the ROM dimension $r$ is analogous to reducing the mesh size $h$ in numerical analysis, the tolerance 
index $k$ for the truncated SVD (which is tied specifically to the data-driven aspect of the LES-ROM closure examined here) has no analogue in classical numerical analysis.}

\subsection{Burgers Equation}
	\label{sec:numerical-results-burgers}

In this section, we investigate the DD-VMS-ROM verifiability in the numerical simulation of the one-dimensional viscous Burgers equation: 
\begin{equation}
\begin{cases}
\displaystyle~~	u_t -\nu u_{xx} + u u_x = 0~,~~~x \in (0,1),~t\in(0,1], \\
\displaystyle~~	u(0,t)= u(1,t) = 0~,~~~t \in (0,1],  \\
\displaystyle~~ u(x,0)=u_0(x) ~, ~~~~ x \in [0,1],
\end{cases}
\label{eqn:burgers}
\end{equation}
with non-smooth initial condition \eqref{eqn:non-smooth-ic}:
\begin{equation}
u_0(x)=\begin{cases}
\displaystyle~ 1, & x \in (0,1/2],\\
~\displaystyle 0, & x \in (1/2,1]. 
\end{cases}
	\label{eqn:non-smooth-ic}
\end{equation}
This test problem has been used in, e.g.,~\cite{ahmed2018stabilized,KV01,xie2018data}.

\paragraph{Snapshot Generation}
We generate the FOM results by using a linear finite element (FE) spatial discretization with mesh size $h=1/2048$, a {backward Euler} time discretization with timestep size {$\Delta t=5\times 10^{-4}$}, and a viscosity coefficient 
{$\nu=10^{-2}$.}
{Due to the parabolic nature of the Burgers equation \eqref{eqn:burgers}, the discontinuity in the initial data \eqref{eqn:non-smooth-ic} is smoothed out as soon as $t>0$. It becomes a (smooth) viscous shock with relatively steep gradient due to the 
small viscosity used, and 
persists for the whole duration of the time integration, i.e., for $t$ in $[0,1]$. 
See also \cite{chen2022shock,Iliescu_al18}, where 
a stochastic version of 
this type of viscous shocks 
is considered within a 
reduced order modeling context.}

\paragraph{ROM Construction}
We run the FOM from $t = 0$ to $t = 1$, {which 
yields a total of 2001 solution snapshots. Since 
the spatial derivatives of the FOM solution 
are involved in the $\boldsymbol \tau^{FOM}$ part of the closure error $\eta(L^2)$ (see \eqref{eqn:tau}) and the initial condition given by \eqref{eqn:non-smooth-ic} is discontinuous, we remove the FOM solution 
in the time interval $[0, 0.01)$, and thus collect a total of 1981 equally spaced snapshots in the time 
interval $[0.01,1]$ to generate the ROM basis functions.} 
To train the DD-VMS-ROM closure operator $\widetilde{\boldsymbol A}$, 
we use FOM data on the same time interval {$[0.01, 1]$}.
We also test the DD-VMS-ROM on the time interval {$[0.01, 1]$. 
That is, each ROM is initialized at $t=0.01$ using the projected FOM data and run up to $t=1$, and the ROM error $\mathcal{E} (L^2)$ in~\eqref{eqn:l2-error-v} and the closure error $\eta (L^2)$ in \eqref{eq:tau-cal-v} are both computed over the time interval $[0.01, 1]$}.
Thus, we consider the reconstructive regime. {The ROMs are 
integrated with the backward Euler time discretization and 
the same timestep size as that used for the FOM.}

\paragraph{Numerical Results}
{We begin by presenting the results obtained for the first type of experiments 
outlined in Section~\ref{sec:assessment-results}. That is, we fix the ROM dimension $r$, and examine how $\mathcal{E} (L^2)$ in~\eqref{eqn:l2-error-v}, which measures the DD-VMS-ROM error, and $\eta (L^2)$ in \eqref{eq:tau-cal-v}, which measures the DD-VMS-ROM closure error, vary as the tolerance 
index $k$ in the truncated SVD used in the data-driven modeling part is varied. Specifically,} we monitor the decaying rate of $\mathcal{E} (L^2)$ with respect to $\eta (L^2)$ as {$k$ is varied}. The results in {Figure~\ref{fig:Burgers-tSVD_err}, for $r=8, 14$, and $20$}, generally show that, as $\eta(L^2)$ decreases {(red curves)}, so does $\mathcal{E}(L^2)$ {(blue curves)}. {
We note, however, that as shown for $r=8$ and $r=14$ in Figure~\ref{fig:Burgers-tSVD_err}, the global minimum of the ROM error $\mathcal{E} (L^2)$ may not be achieved at $k = r^2$, which corresponds to the case when the full SVD is used for constructing the data-driven closure term $\widetilde{A}$; see the caption of Figure~\ref{fig:Burgers-tSVD_err}. 
We also note that 
larger local fluctuations in both 
curves 
are displayed for $r = 14$ and $r=20$, 
which is due to the fact that the condition number of the data matrix $E^\top E$ increases significantly for 
these two values.
Indeed, the condition number of $E^\top E$ is 
$6.35\times 10^6$ for $r = 20$, $1.2\times 10^5$ for $r=14$, and 
$1.6 \times 10^3$ for $r=8$ 
\footnote{{For the case $r=20$, 
for about $5\%$ of the total $400$ possible $k$ values, 
the constrained linear least squares solver \texttt{lsqlin} fails to converge. 
These $k$ values are scattered around $k=300$. We did not include the corresponding $\mathcal{E} (L^2)$ and $\eta(L^2)$ data in Figure~\ref{fig:Burgers-tSVD_err} and 
we also excluded them when computing the corresponding linear regression slope presented in Figure~\ref{fig:burgers-constrained-linear-lr-verifiability}.}}.}

\begin{figure}[htb!]
    \centering
    \includegraphics[width=\textwidth]{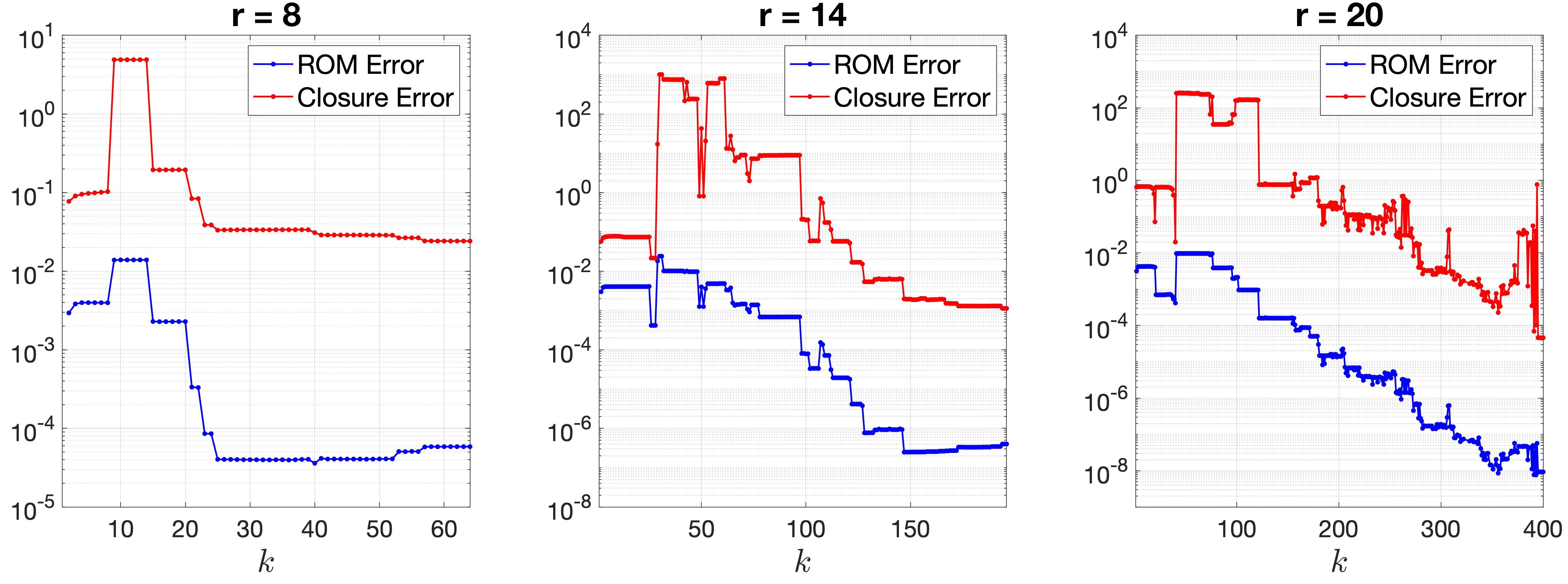}
    \caption{{Burgers equation \eqref{eqn:burgers}, reconstructive regime:
	$\mathcal{E} (L^2)$ and $\eta (L^2)$ for three fixed $r$ values and different tolerance 
	index $k$ values in the truncated SVD. Recall that $\mathcal{E} (L^2)$ and $\eta (L^2)$ are defined by \eqref{eqn:l2-error-v} and \eqref{eq:tau-cal-v}, respectively. 
	As mentioned in Section~\ref{sec:numerical-implementation}, the tolerance values in the truncated SVD 
	take the form of the truncation 
	index $k$, which is the index of the lowest singular value retained in the matrix $\widetilde{\Sigma}$ constructed in step 4 of Algorithm~\ref{alg:truncated-SVD}. For an $r$-dimensional ROM, the matrix $E$ in Algorithm~\ref{alg:truncated-SVD} is of dimension $Mr\times r^2$; cf.~\eqref{eqn:linear-ls-1}. Thus, the tolerance index $k$ can take values between $1$ and $r^2$. As a result, there are $r^2$ data points in each of the three panels for both $\mathcal{E} (L^2)$ (blue curve) and $\eta (L^2)$ (red curve).}
	}
    \label{fig:Burgers-tSVD_err}
\end{figure}

{With the 
$k$-dependence data available, we turn now to 
examining the relation \eqref{Eq_scaling_law} for fixed $r$ values while $k$ is varied. For this purpose, in Figure~\ref{fig:burgers-constrained-linear-lr-verifiability}, we plot the corresponding linear regression (LR) slope. 
We note that the LR slopes shown in Figure~\ref{fig:burgers-constrained-linear-lr-verifiability} are computed based on those $(\mathcal{E} (L^2)$, $\eta (L^2)$) data pairs for which $\eta (L^2) \le 100$ since most of the data pairs are aggregated below that threshold and, more importantly, the cases with small $\eta (L^2)$ are those of practical interest. 
The results in Figure~\ref{fig:burgers-constrained-linear-lr-verifiability} show that \eqref{Eq_scaling_law} holds with $\alpha$ either greater than $1$ or just slightly below $1$. 
}

\begin{figure}[H]
    \centering    
    \includegraphics[width=\textwidth]{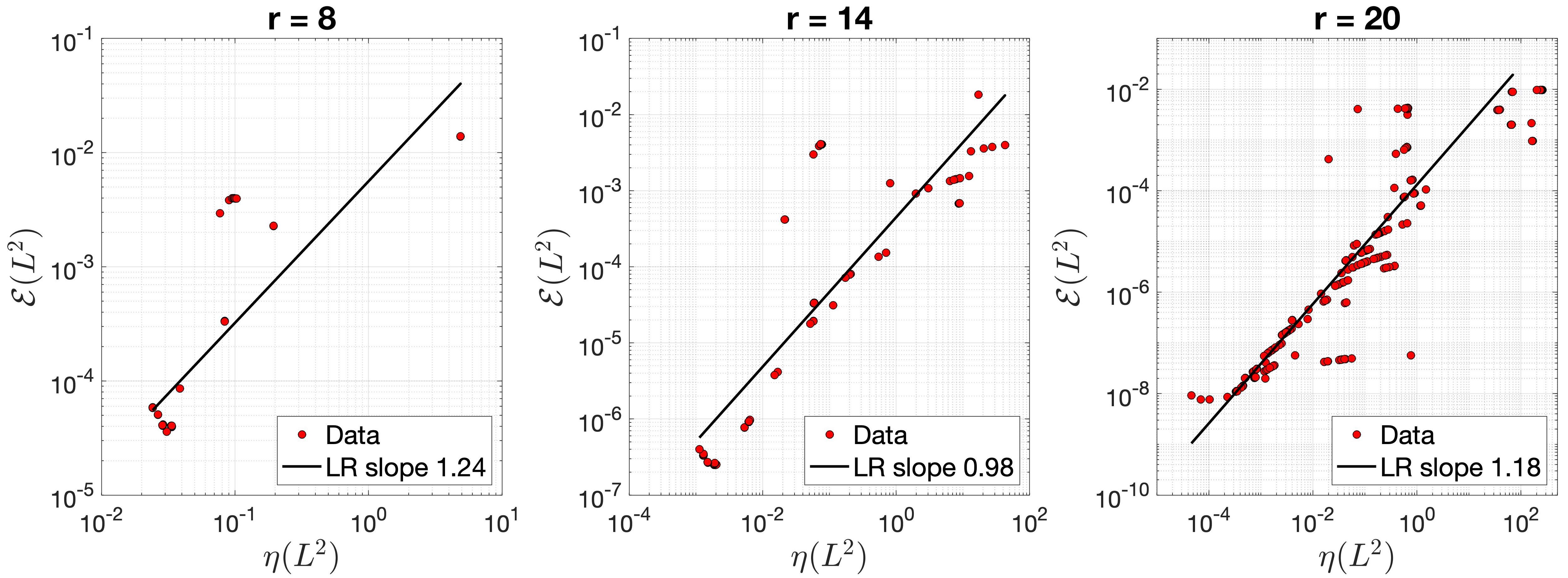}      
    \caption{
	Burgers equation \eqref{eqn:burgers}, reconstructive regime:
	linear regression for $\mathcal{E} (L^2)$ and $\eta (L^2)$ for fixed $r$ values and different tolerance values in the truncated SVD.
	{The 
	red dots in each panel correspond to the data points 
	$(\mathcal{E} (L^2), \eta (L^2))$ shown in the corresponding panel in Figure~\ref{fig:Burgers-tSVD_err}. The linear regression for $\mathcal{E} (L^2)$ in terms of $\eta (L^2)$ in each panel is indicated by the solid black line.}
    }
    \label{fig:burgers-constrained-linear-lr-verifiability}
\end{figure}

{
Next, we consider 
the other type of experiments, in which we vary $r$, and for each $r$ we pick 
the corresponding $k$ that minimizes $\mathcal{E} (L^2)$. These 
results are 
plotted in Figure~\ref{fig:Burgers-err_varying_r}, which shows again that \eqref{Eq_scaling_law} holds with $\alpha \ge 1$, this time 
when $r$ is varied.}
\begin{figure}[H]
    \centering
    \includegraphics[width=\textwidth]{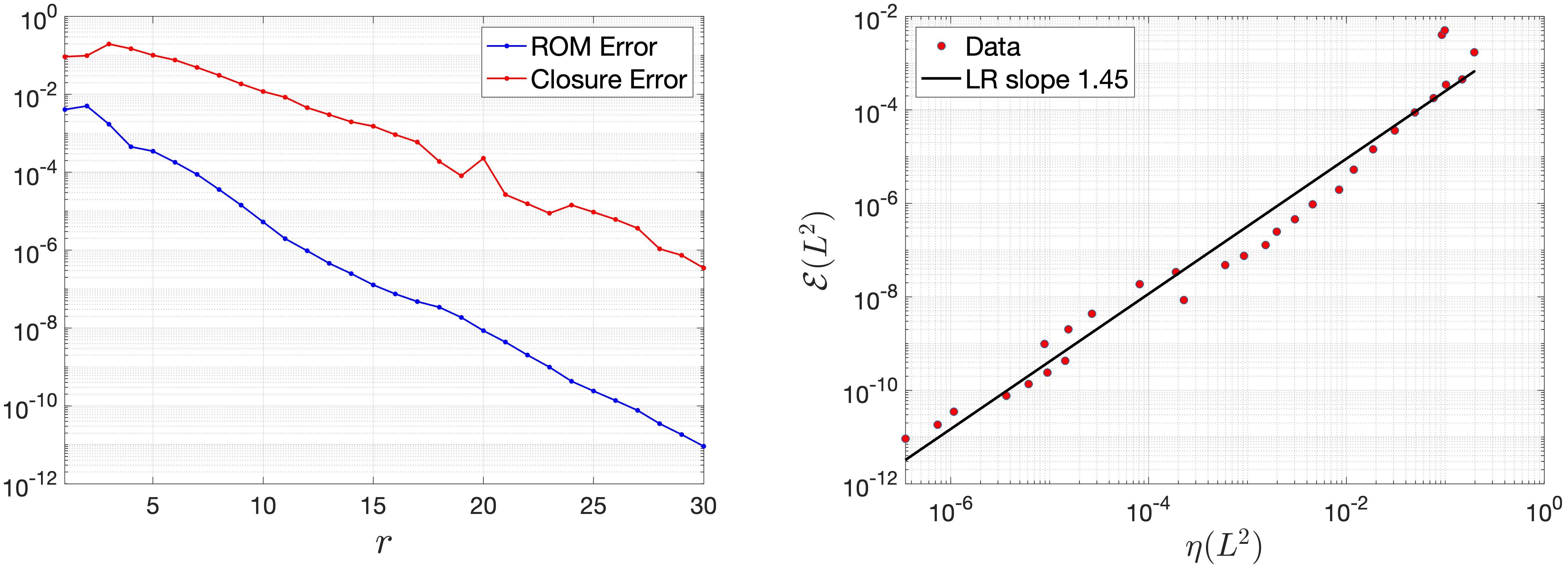}
    \caption{{Burgers equation \eqref{eqn:burgers}, reconstructive regime:
	$\mathcal{E} (L^2)$ and $\eta (L^2)$ as $r$ increases. 
	For each $r$, the tolerance index 
	$k$ in the truncated SVD is chosen 
	to minimize the corresponding ROM error $\mathcal{E} (L^2)$.}
    }
    \label{fig:Burgers-err_varying_r}
\end{figure}

{Overall, the results in this section provide strong numerical support to the theoretical understanding put forth in Theorem~\ref{theorem:main-theorem} in the Burgers equation setting.}

\subsection{Flow Past A Cylinder}
	\label{sec:numerical-results-nse}
	
In this section, we investigate the DD-VMS-ROM verifiability in the numerical simulation of a 2D channel flow past a circular cylinder at Reynolds numbers $Re=100$ and $Re=1000$. 
This test problem has been used in, e.g.,~\cite{mohebujjaman2019physically,mou2021data,xie2018data}.

\paragraph{Computational Setting}
As a mathematical model, we use the NSE~\eqref{eqn:nse-1}--\eqref{eqn:nse-2}.
The computational domain is a $2.2\times 0.41$ rectangular channel with a 
cylinder of radius $0.05$, centered at $(0.2,0.2)$, see Figure~\ref{cyldomain}.  
\begin{figure}[H]
    \centering
    \includegraphics{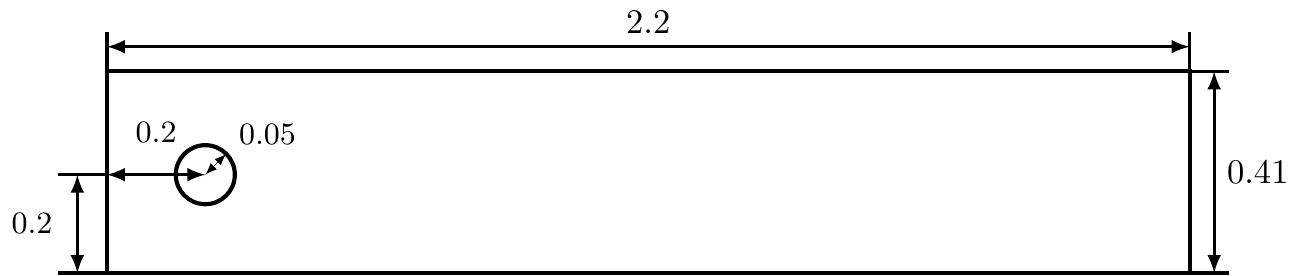}
    
	\caption{\label{cyldomain} 
		Geometry of the flow past a circular cylinder numerical experiment.
		}
\end{figure}

We  prescribe no-slip boundary conditions on the walls and cylinder, and the following inflow and outflow profiles~\cite{john2004reference,mohebujjaman2019physically,rebholz2017improved}:
\begin{align}
u_{1}(0,y,t)&=u_{1}(2.2,y,t)=\frac{6}{0.41^{2}}y(0.41-y), \\ u_{2}(0,y,t)&=u_{2}(2.2,y,t)=0,
\end{align} 
where $\bu=\langle u_1, u_2 \rangle$.  
There is no forcing and the flow starts from rest.

\paragraph{Snapshot Generation}
For the spatial discretization, we use the pointwise divergence-free, LBB stable $(P_2, P_1^{disc})$ Scott-Vogelius finite element pair on a barycenter refined regular triangular mesh~\cite{john2016divergence}. 
The mesh yields $103K$ ($102962$) velocity and $76K$ ($76725$) pressure degrees of freedom. 
We use the 
linearized BDF2 temporal discretization and a time step size $\Delta t=0.002$ for both FOM and ROM time discretizations. 
On the first time step, we use a backward Euler scheme so that we have the two initial time step solutions required for the BDF2 scheme.

\paragraph{ROM Construction}

The FOM simulations 
{settle down to periodic dynamics}
at different time instances for the two Reynolds numbers used in the numerical investigation:
For $Re = 100$ after $t = 5$, 
and for $Re=1000$ after $t=13$.
To construct the ROM basis functions, we use $10$ {time units} of FOM data.
Thus, to ensure a fair comparison of the numerical results at different Reynolds numbers, we collect FOM snapshots on the following time intervals:
For $Re=100$ from $t=7$ to $t=17$, 
and for $Re=1000$ from $t=13$ to $t=23$.

To train the DD-VMS-ROM closure operator $\tA$, we use FOM  data for one period.
The period length of the 
{FOM dynamics}
is different for the two different Reynolds numbers:
From $t = 7$ to $t = 7.332$ for $Re = 100$, 
and from $t = 13$ to $t = 13.268$ for $Re = 1000$. 
Thus, we collect $167$ snapshots for $Re = 100$, and $135$ snapshots for $Re = 1000$.


\subsubsection{Numerical Results for  $Re=100$} 
	\label{sec:numerical-results-nse-re100}

{In Figure~\ref{fig:re100-tSVD_err}}, for three different $r$ values, we {plot} $\mathcal{E} (L^2)$ in~\eqref{eqn:l2-error-v}, which measures the DD-VMS-ROM error, and $\eta (L^2)$ in \eqref{eq:tau-cal-v}, which measures the DD-VMS-ROM closure error. To compute $\mathcal{E} (L^2)$ and $\eta (L^2)$, we fix the $r$ value and decrease the tolerance {
index $k$} in the truncated SVD, which is used in the data-driven modeling part. As the tolerance decreases, we monitor the decaying rate of $\mathcal{E} (L^2)$ with respect to $\eta (L^2)$.  The results {in Figure~\ref{fig:re100-tSVD_err}}, for $r=4, 6$, and $8$, generally show that, as $\eta (L^2)$ decreases, so does $\mathcal{E} (L^2)$. {
We note that, in each panel, the minimal $\mathcal{E} (L^2)$ value is actually achieved at $k = r^2$, 
i.e.,~when the full SVD is used in constructing the closure term $\widetilde{A}$. This is due to the fact that for all the $r$ values  considered, the condition number of the corresponding data matrix $E^\top E$ is always below 
$10^3$.
The same observation is true for the 
$Re=1000$ test case presented in Section~\ref{sec:numerical-results-nse-re1000}. 
}

\begin{figure}[H]
    \centering
    \includegraphics[width=\textwidth]{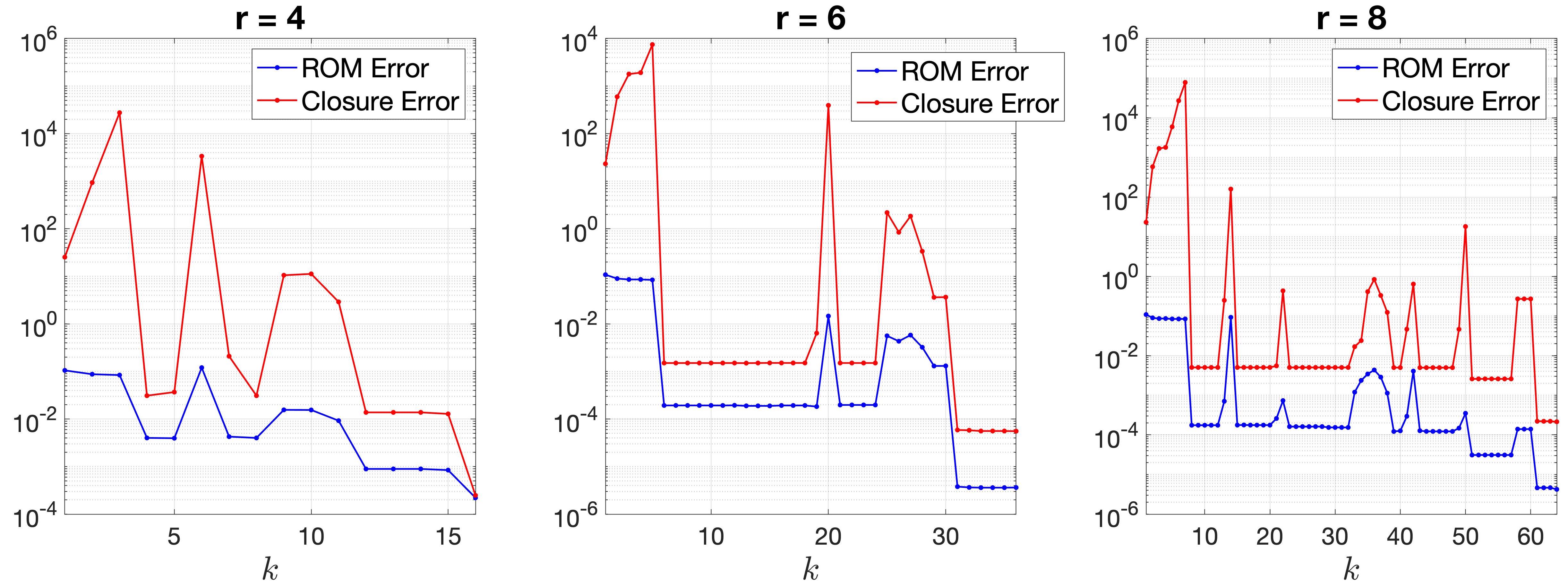}
    \caption{{Flow past a cylinder, $Re=100$, reconstructive regime:	
    $\mathcal{E} (L^2)$ and $\eta (L^2)$ values for fixed $r$ values and different tolerance 
    index $k$ values in the truncated SVD.}
    }
    \label{fig:re100-tSVD_err}
\end{figure}

In Figure~\ref{fig:lr-order-re100-constrained}, for $r=4, 6$, and $8$, we plot the LR slope for $\mathcal{E}(L^2)$ with respect to $\eta(L^2)$.
For $r=4$, the LR slope is $0.56$, for $r=6$ the LR slope is $0.99$, and for $r=8$ the LR slope is $1.03$.
These results indicate an almost linear correlation between $\mathcal{E}(L^2)$ and $\eta(L^2)$, {again in agreement 
with \eqref{Eq_scaling_law} with $\alpha = 1$, except for $r = 4$. One possible explanation for the $r=4$ case 
is that, due to the low-dimensionality of the ROM,  we 
do not have sufficient data points to 
accurately estimate the LR slope. 
}

{When we vary $r$ 
and choose the tolerance index 
$k$ in the truncated SVD to minimize the corresponding $\mathcal{E} (L^2)$, the 
results are shown in Figure~\ref{fig:re100_err_varying_r}. This figure shows again that \eqref{Eq_scaling_law} holds with $\alpha = 1$ in 
the varying $r$ setting.} 

Overall, the results in Figures~\ref{fig:re100-tSVD_err}--\ref{fig:re100_err_varying_r} 
support the theoretical results in Theorem~\ref{theorem:main-theorem}.

\begin{figure}[H]
    \centering
    \includegraphics[width=\textwidth]{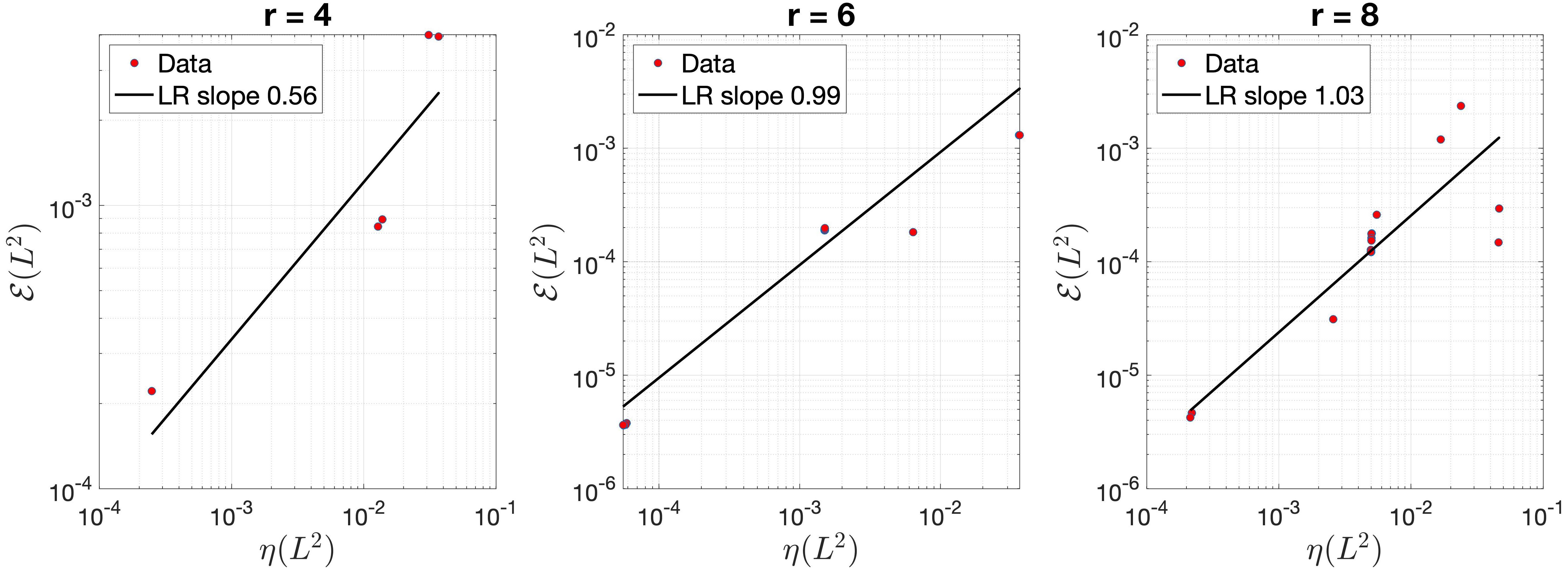}
    \caption{
	Flow  past  a  cylinder, $Re=100$,  reconstructive  regime:
	linear regression for $\mathcal{E} (L^2)$ and $\eta (L^2)$ for fixed $r$ values and different tolerance values in the truncated SVD.
    }
    \label{fig:lr-order-re100-constrained}
\end{figure}

\begin{figure}[H]
    \centering
    \includegraphics[width=\textwidth]{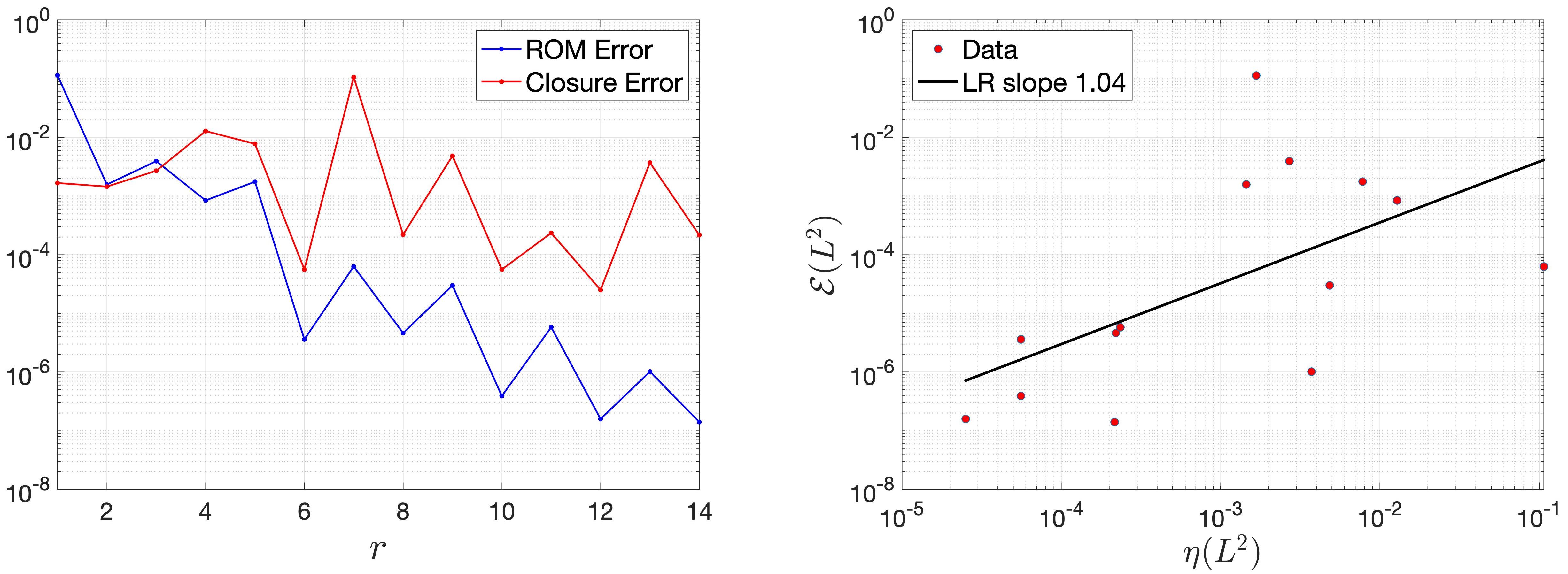}
    \caption{{Flow past a cylinder, $Re=100$, reconstructive regime:	
    $\mathcal{E} (L^2)$ and $\eta (L^2)$ values as $r$ 
    increases.  
    For each $r$, the tolerance index 
    $k$ in the truncated SVD is chosen 
    to minimize the corresponding ROM error $\mathcal{E} (L^2)$.}
    }
    \label{fig:re100_err_varying_r}
\end{figure}

\subsubsection{Numerical Results for  $Re=1000$} 
	\label{sec:numerical-results-nse-re1000}

{In Figure~\ref{fig:re1000-tSVD_err}}, for three different $r$ values, we {plot} $\mathcal{E} (L^2)$ in~\eqref{eqn:l2-error-v}, which measures the DD-VMS-ROM error, and $\eta (L^2)$ in \eqref{eq:tau-cal-v}, which measures the DD-VMS-ROM closure error.
To compute $\mathcal{E} (L^2)$ and $\eta (L^2)$, we fix the $r$ value and decrease the tolerance in the truncated SVD, which is used in the data-driven modeling part. 
As the tolerance decreases, we monitor the decaying rate of $\mathcal{E} (L^2)$ with respect to $\eta (L^2)$. 
The results in {Figure~\ref{fig:re1000-tSVD_err}}, for $r=4, 6$, and $8$, generally show that, as $\eta (L^2)$ decreases, so does $\mathcal{E} (L^2)$. 

\begin{figure}[H]
    \centering
    \includegraphics[width=\textwidth]{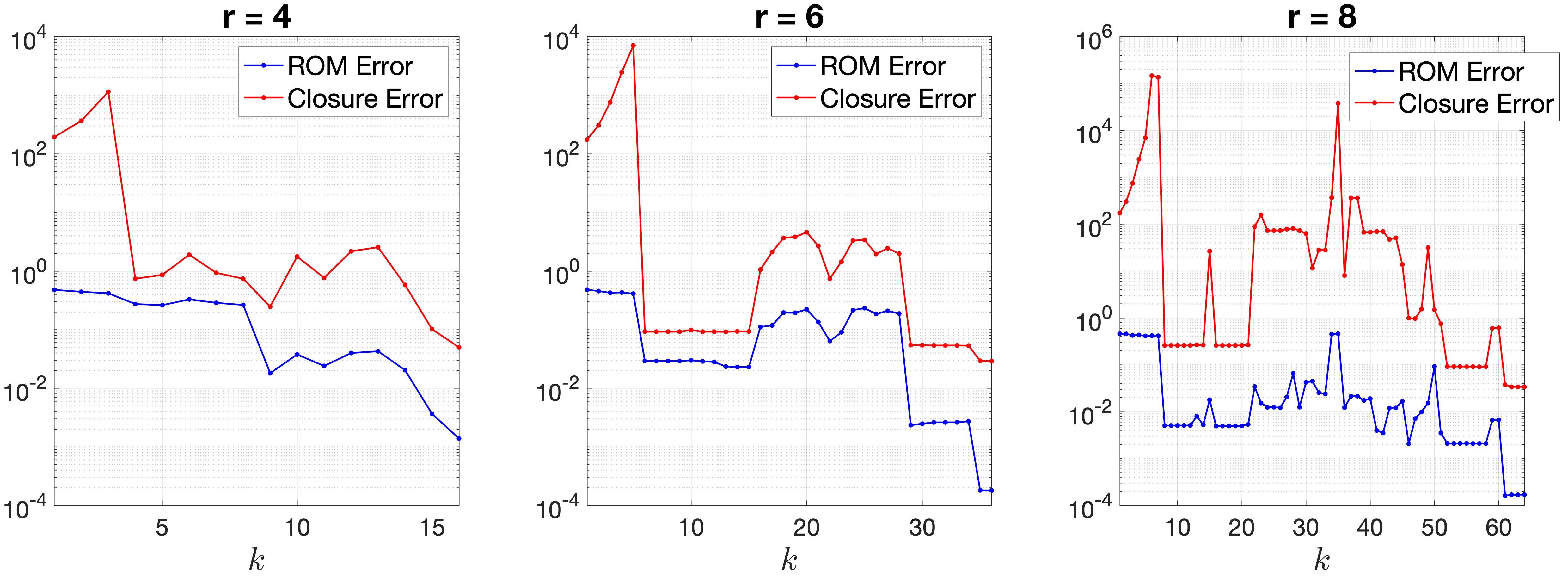}
    \caption{{Flow past a cylinder, $Re=1000$, reconstructive regime:	
    $\mathcal{E} (L^2)$ and $\eta (L^2)$ values for fixed $r$ values and different tolerance 
    index $k$ values in the truncated SVD.}
    }
    \label{fig:re1000-tSVD_err}
\end{figure}

In Figure~\ref{fig:lr-order-re1000-constrained}, for $r=4, 6$, and $8$, we plot the LR slope for $\mathcal{E}(L^2)$ with respect to $\eta(L^2)$.
For $r=4$, the LR slope is $1.71$, for $r=6$ the LR slope is $2.07$, and for $r=8$ the LR slope is $1.00$.
These results indicate that {$\mathcal{E}(L^2)$ decays at least linearly as $\eta(L^2)$ is reduced, again in agreement 
with \eqref{Eq_scaling_law} with $\alpha \ge 1$.}

{When we vary $r$ 
and choose the tolerance 
index $k$ in the truncated SVD to minimize the corresponding $\mathcal{E} (L^2)$, the 
results are shown in Figure~\ref{fig:re1000_err_varying_r}. This figure shows again that \eqref{Eq_scaling_law} holds with $\alpha \ge 1$ in this varying $r$ setting.} 

Overall, the results in Figures~\ref{fig:re1000-tSVD_err}--\ref{fig:re1000_err_varying_r} 
support the theoretical results in Theorem~\ref{theorem:main-theorem}, {
yielding} the same conclusion as that in Section~\ref{sec:numerical-results-nse-re100}.

\begin{figure}[H]
    \centering
    \includegraphics[width=\textwidth]{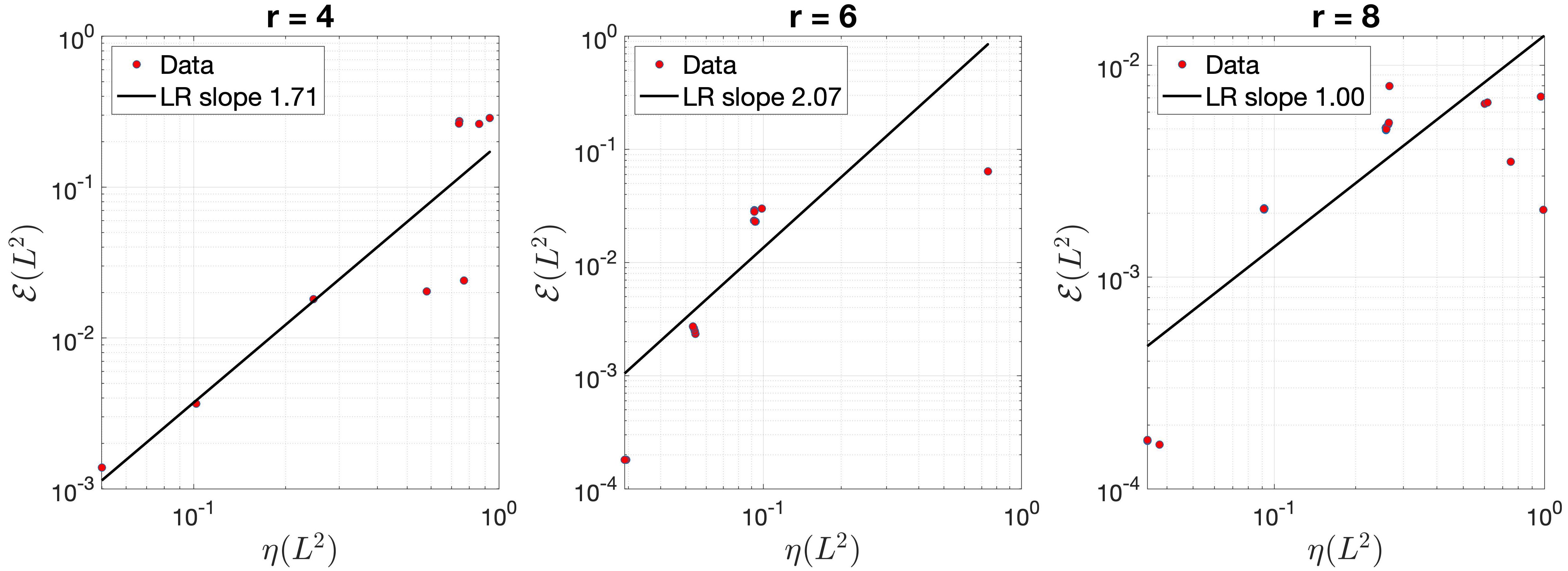}
    \caption{
    Flow  past  a  cylinder, $Re=1000$,  reconstructive  regime:
	linear regression for $\mathcal{E} (L^2)$ and $\eta (L^2)$ for fixed $r$ values and different tolerance values in the truncated SVD.
    }
    \label{fig:lr-order-re1000-constrained}
\end{figure}

\begin{figure}[H]
    \centering
    \includegraphics[width=\textwidth]{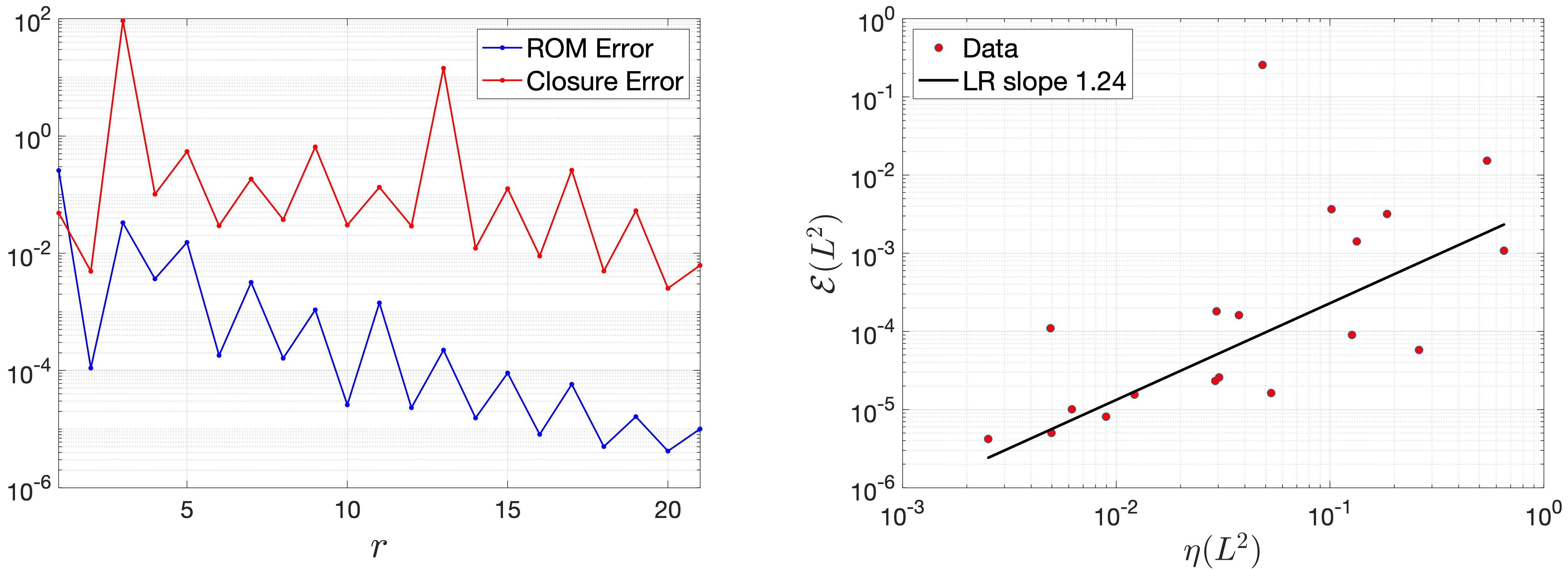}
    \caption{{Flow past a cylinder, $Re=1000$, reconstructive regime:	
    $\mathcal{E} (L^2)$ and $\eta (L^2)$ values as $r$ 
    increases. 
    For each $r$, the tolerance 
    index $k$ in the truncated SVD is chosen 
    to minimize the corresponding ROM error $\mathcal{E} (L^2)$.}
    }
    \label{fig:re1000_err_varying_r}
\end{figure}

\section{Conclusions and Future Work}
    \label{sec:conclusions}

Over the last two decades, a plethora of ROM closure models have been developed for reduced order modeling of convection-dominated flows.
Various ROM closure models have been constructed by using physical insight, mathematical arguments, or data.
Although these ROM closure models are built by using different arguments, they are constructed by using the same {\it heuristic} algorithm:
(i) In the offline stage, the ROM closure model is built so that it is as close as possible (in some norm) to the ``true" ROM closure term.
(ii) In the online stage, one needs to check whether the ROM closure model yields a ROM solution that is as close as possible to the filtered FOM solution.
If the ROM solution is an accurate approximation of the filtered FOM solution, the ROM closure model is deemed accurate.
This heuristic algorithm is the most popular approach used in assessing the success of the current ROM closure models.
However, a natural question is whether one can actually {\it prove} anything about these ROM closure models.
For example, can one prove that an accurate ROM closure model (constructed in the offline phase) yields an accurate ROM solution (in the online phase)?

In this paper, we took a step in this direction and answered the above question by extending the verifiability concept from classical LES to a ROM setting.
Specifically, we defined a ROM closure model as verifiable if the ROM error is bounded (in some norm) by the ROM closure model error.
Furthermore, we proved that a recently introduced data-driven ROM closure model (i.e., the DD-VMS-ROM~\cite{mou2021data,xie2018data}) is verifiable.
Finally, we showed numerically that the DD-VMS-ROM closure is verifiable.
Specifically, in the numerical simulation of the one-dimensional Burgers equation and the two-dimensional flow past a circular cylinder at Reynolds numbers $Re=100$ and $Re=1000$, we showed that by reducing the error in the ROM closure term, we can achieve a decrease in the ROM error, as predicted by the theoretical results.

There are several natural research directions that can be pursued in the quest to lay mathematical foundations for ROM closure models.
For example, one could investigate the verifiability of (functional, structural, or data-driven) ROM closure models that are different from the DD-VMS-ROM investigated in this paper.
One could also extend the verifiability concept to ROM closures that are built from experimental data.
In that case, one could replace the high-dimensional ``truth" solution used in this paper with the experimental solution interpolated onto a discrete mesh.
{
Another potential research direction is the investigation of different norms (e.g., the $H^1$ norm) in the least squares problem~\eqref{eqn:least-squares}, verifiability definition (i.e., Definition~\ref{definition:DD-VMS-ROM-verifiability}), and verfiability theorem (i.e., Theorem~\ref{theorem:main-theorem}).
}
Finally, one could consider other mathematical concepts that are used in classical LES (see, e.g.,~\cite{BIL05}) and extend them to a ROM setting.


\section*{Acknowledgements}

{We thank the reviewers for the insightful comments and suggestions, which have significantly improved the paper.}
The work of the first, second, and sixth authors was supported by NSF through grants DMS-2012253 and CDS\&E-MSS-1953113.
The third author acknowledges the support by NSF through grant DMS-2108856. 
The fifth author acknowledges the support by European Union Funding for Research and Innovation -- Horizon 2020 Program -- in the framework of European Research Council Executive Agency: Consolidator Grant H2020 ERC CoG 2015 AROMA-CFD project 681447 ``Advanced Reduced Order Methods with Applications in Computational Fluid Dynamics,'' the PRIN 2017  ``Numerical Analysis for Full and Reduced Order Methods for the efficient and accurate solution of complex systems governed by Partial Differential Equations'' (NA-FROM-PDEs), and the INDAM-GNCS project ``Tecniche Numeriche Avanzate per Applicazioni Industriali.''

\section*{Data Availability}
The datasets generated during 
the current study are available from the corresponding author on reasonable request.

\bibliographystyle{spmpsci}      

\bibliography{birgul,traian,honghu}

\end{document}

%% file: notation.tex
\def\PP{{{\rm l}\kern - .15em {\rm P} }}
\def\PN2{{\PP_{N}-\PP_{N-2}}}

\newcommand{\cO}{\mathcal{O}}

\newcommand{\bphi}{\boldsymbol{\varphi}}

\newcommand{\btau}{\boldsymbol{\tau}}

\newcommand{\ba}{\boldsymbol{a}}

\newcommand{\bb}{\boldsymbol{b}}

\newcommand{\be}{\boldsymbol{e}}

\newcommand{\bH}{\boldsymbol{H}}

\newcommand{\bu}{\boldsymbol{u}}

\newcommand{\bw}{\boldsymbol{w}}

\newcommand{\bx}{\boldsymbol{x}}
\newcommand{\bX}{\boldsymbol{X}}

\newcommand{\tA}{\tilde{A}}

\newcommand{\deleted}[1]{{}}
